\documentclass[11pt]{article}

\usepackage[a4paper,margin=1in]{geometry}
\usepackage{amsmath,amssymb,amsthm,mathtools}
\usepackage{booktabs}
\usepackage{graphicx}
\usepackage{subcaption}
\usepackage{enumitem}
\usepackage{microtype}
\usepackage{hyperref}
\usepackage{cleveref}

\hypersetup{
  colorlinks=true,
  linkcolor=blue,
  citecolor=blue,
  urlcolor=blue
}

\newtheorem{theorem}{Theorem}[section]
\newtheorem{proposition}[theorem]{Proposition}
\newtheorem{lemma}[theorem]{Lemma}
\newtheorem{corollary}[theorem]{Corollary}
\theoremstyle{definition}
\newtheorem{definition}[theorem]{Definition}
\theoremstyle{remark}
\newtheorem{remark}[theorem]{Remark}

\title{Contour Hankel dynamics and indicator fields for the Riemann $\Xi$-function}
\author{
  Zhiliang Deng\thanks{School of Mathematical Science, University of Electronic Science and Technology of China, China. Email: \texttt{dengzhl@uestc.edu.cn}.}
  \and
  Xiaomei Yang\thanks{School of Mathematics, Southwest Jiaotong University, China. Email: \texttt{yangxiaomath@swjtu.edu.cn}.}
  \and
  Huazhong L\"u\thanks{School of Mathematical Science, University of Electronic Science and Technology of China, China. Email: \texttt{lvhz@uestc.edu.cn}.}
}
\date{}

\begin{document}
\maketitle

\begin{abstract}
We develop a moving-contour Hankel framework for encoding local zero configurations of the Riemann $\Xi$-function. Weighted contour integrals of the logarithmic derivative $\Xi'/\Xi$, expressed in a holomorphic coordinate associated with the contour, are identified with the power moments of a finite atomic measure supported at the coordinate images of the enclosed zeros. This representation yields exact zero-free and rank criteria and, for conjugation-compatible contour-coordinate pairs, an inertia formula: once the matrix order is at least the number of distinct coordinate nodes, the negative index equals the number of distinct nonreal conjugate pairs. Consequently, the Riemann hypothesis admits a local finite-dimensional Hankel-positivity formulation, although establishing this positivity independently of the zero set remains unresolved. As the contour moves, the Hankel matrix evolves by a continuous congruence flow between zero crossings and undergoes finite-rank jumps at crossing events. An isolated zero produces a signed rank-one jump, whereas a nonreal conjugate pair produces a rank-two indefinite event in a real-axis circular scan that meets the pair. Removing the continuous coordinate drift yields a piecewise-constant matrix process from which crossing coordinates, multiplicities, and zero locations can be recovered. Numerical experiments validate the contour quadrature, indicator fields, continuous flow, crossing signatures, and recovery procedure.
\end{abstract}

\section{Introduction}
\label{sec:introduction}

Let
\begin{equation}
\xi(s)
=
\frac12 s(s-1)\pi^{-s/2}
\Gamma\left(\frac{s}{2}\right)\zeta(s),
\qquad
\Xi(t)
=
\xi\left(\frac12+\mathrm i t\right),
\label{eq:xi-Xi-def}
\end{equation}
where $\zeta$ is the Riemann zeta function. The functional equation and Schwarz symmetry imply
\begin{equation}
\Xi(-t)=\Xi(t),
\qquad
\Xi(\overline t)=\overline{\Xi(t)}.
\label{eq:Xi-symmetries}
\end{equation}
The Riemann hypothesis (RH) is equivalent to the assertion that every zero of the entire function $\Xi$ is real. Its logarithmic derivative
\begin{equation}
G(t)
=
\frac{\Xi'(t)}{\Xi(t)}
\label{eq:G-def}
\end{equation}
is meromorphic, with poles at the zeros of $\Xi$ and residues equal to their multiplicities. Contour integrals of $G$ therefore encode local information about the zero configuration through the argument principle and its weighted variants \cite{Edwards1974}.

Weighted logarithmic-derivative integrals have long been used to compute power sums of enclosed zeros and to recover their locations and multiplicities by algebraic moment methods \cite{DelvesLyness1967}. Closely related contour moments arise in numerical eigensolvers, where they extract eigenvalues lying in a prescribed region \cite{SakuraiSugiura2003}, while systems-theoretic interpretations connect block Hankel contour data with realization, rational interpolation, and Loewner constructions \cite{BrennanEmbreeGugercin2023}. In these approaches, the contour is typically fixed and serves as a device for extracting spectral data from a prescribed region.

The present work instead treats the contour as a moving local observation window. For a contour $\Gamma$ and an associated holomorphic coordinate $\phi_\Gamma$, we form weighted contour moments of $G$ and assemble them into a finite Hankel matrix $H_m(\Gamma;\phi_\Gamma)$. The contour determines which zeros are observed, while the coordinate determines how the enclosed zeros are represented. As the contour-coordinate pair moves through the $t$-plane, the resulting matrices form a trajectory that records both the continuous variation of the coordinate representation and the discrete entry or exit of zeros. The underlying contour-moment and Hankel reconstruction mechanism is closely related to our recent work on the recovery and certification of poles of meromorphic functions \cite{ContourCountPoleClusters2026,YangDeng2026PoleCertification}. The essential new feature here is dynamical: the Hankel matrix is treated as a parameter-dependent state whose continuous evolution is interrupted by finite-rank zero-crossing events.

The residue theorem identifies the contour moments with the power moments of a finite atomic measure supported at the coordinate images of the enclosed zeros. This representation yields the static algebraic structure of the method: the Hankel matrix vanishes exactly when the contour interior is zero-free, and its stabilized rank equals the number of distinct coordinate nodes. For a conjugation-compatible contour-coordinate pair, the matrix is real symmetric and admits an exact inertia classification. Once the matrix order is sufficiently large, each real node contributes one positive direction, whereas each nonreal conjugate pair contributes one positive and one negative direction. Consequently, the negative index counts the distinct nonreal conjugate pairs represented by the contour coordinate. This inertia formula gives a local finite-dimensional Hankel-positivity formulation of RH for real-centered circular contours. The statement is an equivalence rather than a proof: the unresolved problem is to derive the required positivity directly from the analytic or arithmetic structure of $\Xi$, without using prior information about its zeros.

The main dynamical results describe how this matrix representation changes when the contour-coordinate pair moves. Between zero crossings, normalized circular contours generate a finite-dimensional Lyapunov-type equation whose solution is a congruence flow; hence rank is preserved, and inertia is preserved whenever the matrices are real symmetric. At an isolated crossing, the atomic measure gains or loses one weighted point mass, producing a signed rank-one jump, while simultaneous crossings produce finite-rank updates. In a real-axis circular scan, a real zero gives a rank-one signed semidefinite event, whereas a nonreal conjugate pair met by the scan gives a rank-two indefinite event. Conjugating by the known continuous flow removes the coordinate drift and converts the trajectory into a piecewise-constant matrix process. Under an isolated-event assumption, its jumps determine the crossing direction, multiplicity, and local coordinate, and for circular contours the zero location follows directly. Multiple simultaneous events can be recovered from the jump moments through a finite Prony-type system, with conditioning governed by the separation of the crossing coordinates.

The matrix trajectory also generates scalar fields on the contour parameter space. We consider a zero-count field, a Hankel-energy field, and, for real symmetric contour matrices, a positive-semidefiniteness-violation field. These quantities describe the zero information contained in the moving testing region and should not be interpreted as pointwise values attached to the contour center. The numerical experiments validate the computational content of this local framework rather than provide numerical evidence for RH. They examine contour-quadrature accuracy and its deterioration near a crossing, illustrate the three indicator fields, verify the benefit of centered normalization, and test the continuous flow and isolated-event recovery using the first positive zero of $\Xi$.

Several other Hankel or positivity constructions have been studied in connection with RH. The de Bruijn--Newman theory deforms the Fourier representation of $\Xi$ by a heat flow \cite{Polymath2019,RodgersTao2018}. Bickel, Pascoe, and Sargent obtain quantitative real-rootedness information from truncated Hankel positivity for genus-one entire functions \cite{BickelPascoeSargent2023}. Suzuki develops positivity-type criteria through a screw function associated with the zeta function \cite{Suzuki2023}, while Matiyasevich studies special Hankel matrices formed from global coefficient sequences \cite{Matiyasevich2007,Matiyasevich2017}. The matrices considered here are different: they are built from local contour moments, and their dynamics are induced geometrically by moving the observation window relative to a fixed zero set.

The paper is organized as follows. Section~\ref{sec:contour_moment} introduces contour-coordinate pairs, moments, and Hankel matrices. Section~\ref{sec:structural-properties} develops their vanishing, rank, inertia, local positivity, and affine covariance properties. Section~\ref{sec:hankel-dynamics} derives the continuous and impulsive Hankel dynamics and the recovery of crossing events. Section~\ref{sec:indicator-functionals} introduces the scalar indicator fields, Section~\ref{sec:numerics} presents the numerical validation, and the final section summarizes the conclusions and open problem.

\section{Contour-coordinate moments and Hankel matrices}
\label{sec:contour_moment}

\subsection{Contour-coordinate pairs}

Let $\Gamma\subset\mathbb C$ be a positively oriented piecewise smooth Jordan contour, and let $\Omega_\Gamma$ denote the bounded domain enclosed by $\Gamma$. We assume that $\Xi$ is holomorphic in a neighborhood of $\overline{\Omega_\Gamma}$ and that
\begin{equation}
\Xi(t)\neq 0, \qquad t\in\Gamma.
\label{eq:no-zero-on-contour}
\end{equation}
Consequently, the logarithmic derivative
$$
G(t)=\frac{\Xi'(t)}{\Xi(t)}
$$
is well defined on $\Gamma$ and holomorphic in a neighborhood of $\Gamma$. Inside $\Omega_\Gamma$, its poles are precisely the zeros of $\Xi$, with residues equal to their multiplicities.

The contour determines which zeros of $\Xi$ are enclosed, but it does not by itself provide a normalized numerical representation of their locations. To encode the enclosed zero configuration, we introduce an auxiliary holomorphic coordinate function.

\begin{definition}[Contour coordinate function]
A contour coordinate function associated with $\Gamma$ is a function $\phi_\Gamma$ that is holomorphic in a neighborhood of $\overline{\Omega_\Gamma}$. The pair
$(\Gamma,\phi_\Gamma)$ is called a contour-coordinate pair.
\end{definition}

The role of $\phi_\Gamma$ is to assign coordinate values to the zeros enclosed by $\Gamma$. Suppose that the distinct zeros of $\Xi$ in $\Omega_\Gamma$ are
$\tau_1,\ldots,\tau_q$
with respective multiplicities $n_1,\ldots,n_q$.
Their associated coordinate values are defined by
\begin{equation}
u_\ell=\phi_\Gamma(\tau_\ell), \qquad \ell=1,\ldots,q.
\label{eq:coordinate-values}
\end{equation}
Thus, the contour determines which zeros are observed, while the coordinate function determines how their locations are represented.

In general, distinct zeros need not have distinct coordinate values. When
$\phi_\Gamma(\tau_\ell)=\phi_\Gamma(\tau_{\ell'})$ for some $\ell\neq\ell'$, the corresponding zeros are indistinguishable at the level of the contour coordinates.
A coordinate function is said to separate the enclosed zeros if the values $u_1,\ldots,u_q$ are pairwise distinct.

A particularly important choice is obtained from circular contours. For $z\in\mathbb C$ and $\rho>0$, let
\begin{equation}
D_\rho(z)=\{t\in\mathbb C:|t-z|<\rho\},\qquad \Gamma_\rho(z)=\partial D_\rho(z),
\label{eq:circle-contour}
\end{equation}
where $\Gamma_\rho(z)$ is given the positive orientation. We associate with this contour the normalized centered coordinate
\begin{equation}
\phi_{z,\rho}(t)=\frac{t-z}{\rho}.
\label{eq:centered-coordinate}
\end{equation}
This affine coordinate maps $D_\rho(z)$ biholomorphically onto the unit disk. In particular, every zero enclosed by $\Gamma_\rho(z)$ satisfies
$$
|\phi_{z,\rho}(\tau_\ell)|<1.
$$
The normalization therefore controls the magnitude of higher-order moments and avoids the numerical imbalance that may arise from taking large unscaled powers of $t$.

\subsection{Contour moments}

The coordinate values of the enclosed zeros are encoded by a sequence of contour moments.
\begin{definition}[Contour moments]
For $k=0,1,2,\ldots$, define
\begin{equation}
\mu_k(\Gamma;\phi_\Gamma)=\frac{1}{2\pi\mathrm i} \int_\Gamma \phi_\Gamma(t)^k G(t)\,\mathrm dt.
\label{eq:general-moments}
\end{equation}
The sequence
$$
\{\mu_k(\Gamma;\phi_\Gamma)\}_{k\geq 0}
$$
is called the contour moment sequence associated with the contour-coordinate pair $(\Gamma,\phi_\Gamma)$.
\end{definition}

The following residue representation shows that the contour moments depend only on the enclosed zeros, their multiplicities, and their coordinate values.

\begin{lemma}[Residue representation]
\label{lem:power-sum}
Under assumption \eqref{eq:no-zero-on-contour},
\begin{equation}
\mu_k(\Gamma; \phi_\Gamma)=
\sum_{\ell=1}^{q}
n_\ell\phi_\Gamma(\tau_\ell)^k,
\qquad
k=0,1,2,\ldots.
\label{eq:power-sum-identity}
\end{equation}
Equivalently,
\begin{equation}
\mu_k(\Gamma;\phi_\Gamma)=\sum_{\ell=1}^{q}n_\ell u_\ell^k.
\label{eq:power-sum-coordinate}
\end{equation}
\end{lemma}

\begin{proof}
If $\tau_\ell$ is a zero of $\Xi$ of multiplicity $n_\ell$, then locally
$$
\Xi(t)=(t-\tau_\ell)^{n_\ell}h_\ell(t),
$$
where $h_\ell$ is holomorphic and $h_\ell(\tau_\ell)\neq 0$. Hence
$$
G(t)= \frac{\Xi'(t)}{\Xi(t)}=
\frac{n_\ell}{t-\tau_\ell}+\frac{h_\ell'(t)}{h_\ell(t)}.
$$
It follows that
$$
\operatorname{Res}_{t=\tau_\ell}
\bigl(\phi_\Gamma(t)^kG(t)\bigr)=
n_\ell\phi_\Gamma(\tau_\ell)^k.
$$
Summing the residues over all zeros enclosed by $\Gamma$ and applying the residue theorem gives \eqref{eq:power-sum-identity}. Equation \eqref{eq:power-sum-coordinate} then follows from \eqref{eq:coordinate-values}.
\end{proof}

In particular, the zeroth moment is independent of the choice of coordinate function and satisfies
\begin{equation}
\mu_0(\Gamma;\phi_\Gamma)= \sum_{\ell=1}^{q}n_\ell
=N_{\Omega_\Gamma}(\Xi),
\label{eq:count-identity}
\end{equation}
where $N_{\Omega_\Gamma}(\Xi)$ denotes the total number of zeros of $\Xi$ in $\Omega_\Gamma$, counted with multiplicity.

The residue representation admits a natural measure-theoretic interpretation. Define the finite atomic measure
\begin{equation}
\nu_\Gamma=
\sum_{\ell=1}^{q}
n_\ell\delta_{u_\ell},
\label{measure_point}
\end{equation}
where $\delta_{u_\ell}$ denotes the unit point mass at 
$u_\ell=\phi_\Gamma(\tau_\ell)$.
Then
\begin{equation}
\mu_k(\Gamma;\phi_\Gamma)=
\int_{\mathbb C}u^k\,\mathrm d\nu_\Gamma(u),
\qquad
k\geq0.
\label{measure_moment}
\end{equation}
Thus, the contour moments are precisely the power moments of the finite atomic measure $\nu_\Gamma$.

To account explicitly for possible coordinate collisions, let
$\zeta_1,\ldots,\zeta_r$
be the distinct values among $u_1, \ldots, u_q$, and define the aggregated weights
\begin{equation}
\omega_j=
\sum_{\{\ell:\,u_\ell=\zeta_j\}}n_\ell,
\qquad
j=1,\ldots,r.
\label{eq:aggregated-weights}
\end{equation}
Then
\begin{equation}
\nu_\Gamma=
\sum_{j=1}^{r}\omega_j\delta_{\zeta_j},
\qquad
\mu_k(\Gamma;\phi_\Gamma)=
\sum_{j=1}^{r}\omega_j\zeta_j^k.
\label{eq:distinct-support-representation}
\end{equation}
Here $r\leq q$, with $r=q$ precisely when $\phi_\Gamma$ separates the enclosed zeros.

Accordingly, the contour-coordinate pair $(\Gamma,\phi_\Gamma)$ represents the enclosed zero configuration through its coordinate image as a finite atomic measure, while the contour moments record the corresponding power-moment sequence.

\subsection{Contour Hankel matrices}

The contour moment sequence naturally generates a family of Hankel
matrices.
\begin{definition}[Contour Hankel matrix]
For $m\geq1$, define
\begin{equation}
H_m(\Gamma;\phi_\Gamma)=
\bigl(
\mu_{i+j}(\Gamma;\phi_\Gamma)
\bigr)_{i,j=0}^{m-1}.
\label{eq:general-hankel}
\end{equation}
\end{definition}

Thus, $H_m(\Gamma;\phi_\Gamma)$ collects the moments
$\mu_0,\mu_1,\ldots,\mu_{2m-2}$ into an $m\times m$ complex
symmetric matrix. Throughout the paper, the superscript $T$ denotes the ordinary
matrix transpose, rather than the conjugate transpose.

For $u\in\mathbb C$, let
$v_m(u)=(1,u,u^2,\ldots,u^{m-1})^T$.
The power-sum representation of the contour moments in Lemma~\ref{lem:power-sum} gives the following
Vandermonde factorization.

\begin{proposition}[Vandermonde factorization]
\label{prop:vander-factorization}
For every $m\geq1$,
\begin{equation}
H_m(\Gamma;\phi_\Gamma) =
\sum_{\ell=1}^{q}
n_\ell
v_m(u_\ell)v_m(u_\ell)^T.
\label{eq:hankel-sum}
\end{equation}
Equivalently,
\begin{equation}
H_m(\Gamma;\phi_\Gamma)=V_mWV_m^T,
\label{eq:vander-factorization}
\end{equation}
where
$$
V_m=
\begin{bmatrix}
v_m(u_1)&\cdots&v_m(u_q)
\end{bmatrix}
\in\mathbb C^{m\times q},
\qquad
W=
\operatorname{diag}(n_1,\ldots,n_q).
$$

If $\zeta_1,\ldots,\zeta_r$ are the distinct coordinate values among
$u_1,\ldots, u_q$, and if $\omega_j$ denotes the total multiplicity
associated with $\zeta_j$, then
\begin{equation}
H_m(\Gamma;\phi_\Gamma)=\widetilde V_m\widetilde W\widetilde V_m^T,
\label{eq:reduced-vander-factorization}
\end{equation}
where
$$
\widetilde V_m=
\begin{bmatrix}
v_m(\zeta_1)&\cdots&v_m(\zeta_r)
\end{bmatrix},
\qquad
\widetilde W=\operatorname{diag}(\omega_1,\ldots,\omega_r).
$$
\end{proposition}

\begin{proof}
The $(i, j)$-entry of the right-hand side of
\eqref{eq:hankel-sum} is
$\sum_{\ell=1}^{q}n_\ell u_\ell^{i+j}$,
which equals
$\mu_{i+j}(\Gamma;\phi_\Gamma)$ by
\eqref{eq:power-sum-coordinate}.
The reduced factorization
\eqref{eq:reduced-vander-factorization} follows by grouping together
all terms in \eqref{eq:hankel-sum} having the same coordinate value.
\end{proof}

The factorization shows that the contour Hankel matrix is completely
determined by the distinct coordinate values of the enclosed zeros and
their aggregated multiplicities. The reduced form
\eqref{eq:reduced-vander-factorization} will be used below to analyze
the rank and null space of the contour Hankel matrices. Since the
factorizations involve the ordinary transpose rather than the
conjugate transpose, these matrices are generally complex symmetric,
but need not be Hermitian or positive semidefinite.

\section{Structural properties of contour Hankel matrices}
\label{sec:structural-properties}

Throughout this section, we use the reduced atomic representation and Vandermonde factorization introduced in \eqref{eq:distinct-support-representation} and \eqref{eq:reduced-vander-factorization}. Thus $\zeta_1,\ldots,\zeta_r$ are the distinct coordinate nodes and $\omega_1,\ldots,\omega_r$ are their positive aggregated multiplicities.

\subsection{Vanishing and rank structure}

We first record two properties that do not require any assumption on
the location of the enclosed zeros.

\begin{proposition}[Vanishing criterion]
\label{prop:zero-free-iff}
For every $m\geq1$,
\begin{equation}
H_m(\Gamma;\phi_\Gamma)=0
\quad\Longleftrightarrow\quad
\Omega_\Gamma
\text{ contains no zero of }\Xi.
\label{eq:hankel-zero-iff}
\end{equation}
\end{proposition}

\begin{proof}
If $\Omega_\Gamma$ contains no zero of $\Xi$, then the residue sum
in \eqref{eq:power-sum-identity} is empty. Hence all contour moments
vanish and therefore
$H_m(\Gamma;\phi_\Gamma)=0$.
Conversely, if $H_m(\Gamma;\phi_\Gamma)=0$, then its $(0,0)$-entry
satisfies
$\mu_0(\Gamma;\phi_\Gamma)=0$.
By \eqref{eq:count-identity}, this zeroth moment is the number of zeros
of $\Xi$ in $\Omega_\Gamma$, counted with multiplicity. Thus the
interior is zero-free.
\end{proof}

\begin{proposition}[Rank of the contour Hankel matrix]
\label{prop:rank-property}
Let $r$ be the number of distinct support points of $\nu_\Gamma$.
If $m\geq r$, then
\begin{equation}
\operatorname{rank}
H_m(\Gamma;\phi_\Gamma)=r.
\label{eq:rank-r}
\end{equation}
In particular, if $\phi_\Gamma$ separates the $q$ distinct enclosed
zeros, then
$$
\operatorname{rank}
H_m(\Gamma;\phi_\Gamma)=q,
\qquad
m\geq q.
$$
\end{proposition}

\begin{proof}
The reduced factorization \eqref{eq:reduced-vander-factorization} gives
$$
\operatorname{rank}H_m(\Gamma;\phi_\Gamma)\leq r.
$$
Let $V_r$ be the square Vandermonde matrix formed from the first $r$ rows of $\widetilde V_m$. Since the nodes $\zeta_1,\ldots,\zeta_r$ are distinct, $V_r$ is nonsingular. The leading $r\times r$ principal block of $H_m$ is
$$
V_r\widetilde W V_r^T,
$$
which is nonsingular because $\widetilde W$ is positive diagonal. Hence
$$
\operatorname{rank}H_m(\Gamma;\phi_\Gamma)\geq r,
$$
and the claimed equality follows.
\end{proof}

Thus the rank counts distinct coordinate nodes rather than total
multiplicity. A multiple zero changes the corresponding weight but
does not create an additional support point. Likewise, distinct zeros
mapped to the same coordinate value cannot be separated by the moment
sequence.

\subsection{Real-coordinate signatures and inertia}
\label{subsec:real-signatures-inertia}

The contour Hankel matrices are generally complex symmetric rather
than Hermitian. Positivity becomes meaningful when the coordinate
measure is invariant under complex conjugation.

We call a contour-coordinate pair $(\Gamma,\phi_\Gamma)$
conjugation-compatible if $\Omega_\Gamma$ is invariant under complex
conjugation and
\begin{equation}
\phi_\Gamma(\overline t)
=
\overline{\phi_\Gamma(t)}
\label{eq:conjugation-compatible-coordinate}
\end{equation}
throughout a neighborhood of $\overline{\Omega_\Gamma}$. Since
$\Xi(\overline t)=\overline{\Xi(t)}$, the nonreal coordinate nodes
then occur in conjugate pairs with equal positive weights, and
$H_m(\Gamma;\phi_\Gamma)$ is real symmetric.

Real-centered circular contours equipped with
$$
\phi_{x,\rho}(t)=\frac{t-x}{\rho},
\qquad
x\in\mathbb R,
\quad
\rho>0,
$$
are conjugation-compatible.

\begin{proposition}[PSD for real coordinate support]
\label{prop:psd-real-zeros}
Suppose that all support points
$\zeta_1,\ldots,\zeta_r$ of $\nu_\Gamma$ are real. Then
$H_m(\Gamma;\phi_\Gamma)$ is real symmetric positive semidefinite for
every $m\geq1$, and
\begin{equation}
\operatorname{rank}
H_m(\Gamma;\phi_\Gamma)
=
\min\{m,r\}.
\label{eq:real-rank}
\end{equation}
\end{proposition}

\begin{proof}
In this case, $\widetilde V_m$ is real and
$\widetilde W$ is positive diagonal. Hence, for every
$c\in\mathbb R^m$,
$$
c^T H_m(\Gamma;\phi_\Gamma)c
=
(\widetilde V_m^Tc)^T
\widetilde W
(\widetilde V_m^Tc)
\geq0.
$$
Moreover,
$$
H_m
=
\bigl(\widetilde V_m\widetilde W^{1/2}\bigr)
\bigl(\widetilde V_m\widetilde W^{1/2}\bigr)^T,
$$
so its rank equals the rank of the real Vandermonde matrix
$\widetilde V_m$, namely $\min\{m,r\}$.
\end{proof}

The preceding proposition gives only one direction. For sufficiently
large matrix order, conjugation symmetry yields a complete inertia
classification.

\begin{theorem}[Inertia of a conjugation-symmetric contour Hankel matrix]
\label{thm:inertia-conjugation-symmetric}
Suppose that $(\Gamma,\phi_\Gamma)$ is conjugation-compatible and
that the distinct support points of $\nu_\Gamma$ consist of
$r_0$ real nodes
$x_1,\ldots,x_{r_0}\in\mathbb R$
and $c$ nonreal conjugate pairs
$z_1,\overline z_1,\ldots,z_c,\overline z_c$,
$\operatorname{Im}z_j>0$.
Let the positive weights at the real nodes be
$\omega_1,\ldots,\omega_{r_0}$, and suppose that the two nodes
$z_j$ and $\overline z_j$ have the same positive weight
$\eta_j$. Set
$r=r_0+2c$.
Then $H_m(\Gamma;\phi_\Gamma)$ is real symmetric. Moreover, for
every $m\geq r$,
\begin{equation}
\operatorname{In}
H_m(\Gamma;\phi_\Gamma)=\bigl(r_0+c,\ c,\ m-r\bigr),
\label{eq:contour-hankel-inertia}
\end{equation}
where
$\operatorname{In}(M)=
\bigl(n_+(M),n_-(M),n_0(M)\bigr)$
denotes the inertia of a real symmetric matrix $M$.
\end{theorem}

\begin{proof}
For every $k\geq0$,
$\mu_k=
\sum_{i=1}^{r_0}\omega_i x_i^k+
\sum_{j=1}^{c}\eta_j
\left(z_j^k+\overline z_j^{k}\right)
\in\mathbb R$.
Hence $H_m(\Gamma;\phi_\Gamma)$ is real symmetric.

For each nonreal node, write
$v_m(z_j)=a_j+\mathrm i b_j$,
$a_j, b_j\in\mathbb R^m$.
The contribution of the conjugate pair is
\begin{align*}
\eta_j
\left[v_m(z_j)v_m(z_j)^T + v_m(\overline z_j)v_m(\overline z_j)^T\right]=
2\eta_j\bigl(a_ja_j^T-b_jb_j^T\bigr).
\end{align*}
Consequently,
\begin{equation}
H_m(\Gamma;\phi_\Gamma)=B_mJ B_m^T,
\label{eq:real-inertia-factorization}
\end{equation}
where
$$B_m=
\begin{bmatrix}
 v_m(x_1)&\cdots&v_m(x_{r_0})&
 a_1&b_1&\cdots&a_c&b_c
\end{bmatrix}$$
and
$$J=
\operatorname{diag}
\bigl(
\omega_1,\ldots,\omega_{r_0},
2\eta_1,-2\eta_1,\ldots,2\eta_c,-2\eta_c
\bigr).$$

The complex Vandermonde matrix generated by the distinct nodes
$$
x_1,\ldots,x_{r_0},
 z_1,\overline z_1,\ldots,z_c,\overline z_c
$$
has full column rank for $m\geq r$. Replacing each pair of columns
$v_m(z_j),v_m(\overline z_j)$ by
$a_j,b_j$ is an invertible column transformation. Hence $B_m$ has
full column rank $r$.

Choose a thin QR factorization $B_m=QR$, where
$Q\in\mathbb R^{m\times r}$ has orthonormal columns and
$R\in\mathbb R^{r\times r}$ is nonsingular. Then
$$
H_m
=
Q(RJR^T)Q^T.
$$
By Sylvester's law of inertia, $RJR^T$ has the same inertia as
$J$, namely $(r_0+c,c,0)$. The orthogonal complement of the range
of $Q$ contributes $m-r$ zero eigenvalues. This proves
\eqref{eq:contour-hankel-inertia}.
\end{proof}

\begin{corollary}[Local real-support criterion]
\label{cor:local-real-support}
Under the assumptions of
Theorem~\ref{thm:inertia-conjugation-symmetric}, let $r$ be the
number of distinct coordinate nodes. Then, for every $m\geq r$,
\begin{equation}
H_m(\Gamma;\phi_\Gamma)\succeq0
\quad\Longleftrightarrow\quad
\operatorname{supp}\nu_\Gamma\subset\mathbb R.
\label{eq:local-real-support-criterion}
\end{equation}
Moreover,
\begin{equation}
n_-
\bigl(H_m(\Gamma;\phi_\Gamma)\bigr)=c,
\qquad
m\geq r,
\label{eq:negative-index-counts-pairs}
\end{equation}
so the negative index equals the number of distinct nonreal conjugate
pairs represented by the contour coordinate.
\end{corollary}

\begin{proof}
The conclusion follows immediately from
\eqref{eq:contour-hankel-inertia}. The matrix is positive semidefinite
if and only if its negative index is zero, which is equivalent to
$c=0$.
\end{proof}

\begin{corollary}[A nonreal conjugate pair]
\label{cor:nonreal-pair-violation}
Suppose that the coordinate measure consists of a single nonreal
conjugate pair of equal multiplicity $n$, with nodes
$u=a+\mathrm i b$,
$\overline u=a-\mathrm i b$,
$b\neq0$.
Then, for every $m\geq2$,
$\operatorname{In}
H_m(\Gamma;\phi_\Gamma)=(1,1,m-2)$.
In particular, $H_m(\Gamma;\phi_\Gamma)$ is indefinite for every
$m\geq2$. At the smallest admissible order,
$$
H_2(\Gamma;\phi_\Gamma)
=
2n
\begin{pmatrix}
1&a\\
a&a^2-b^2
\end{pmatrix},
$$
and
$$
\det H_2(\Gamma;\phi_\Gamma)
=
-4n^2b^2<0.
$$
Thus $H_2$ is the lowest-order contour Hankel matrix that detects
the nonreal character of the conjugate pair.
\end{corollary}

\begin{proof}
This is the case $r_0=0$, $c=1$, and $r=2$ of
Theorem~\ref{thm:inertia-conjugation-symmetric}. Hence, for every
$m\geq2$,
$$
\operatorname{In}
H_m(\Gamma;\phi_\Gamma)
=
(1,1,m-2).
$$
In particular, $H_m(\Gamma;\phi_\Gamma)$ is indefinite for every
$m\geq2$. At the smallest admissible order $m=2$, direct
calculation gives
$$
H_2(\Gamma;\phi_\Gamma)
=
2n
\begin{pmatrix}
1&a\\
a&a^2-b^2
\end{pmatrix},
$$
and therefore
$$
\det H_2(\Gamma;\phi_\Gamma)
=
-4n^2b^2<0.
$$
\end{proof}

The inertia theorem gives a finite-dimensional local positivity
formulation of the Riemann hypothesis. For $x\in\mathbb R$ and
$\rho>0$, write
$$
\mathcal N_\rho(x)
:=
\mu_0\bigl(\Gamma_\rho(x);\phi_{x,\rho}\bigr)
=
N_{D_\rho(x)}(\Xi)
$$
for the number of zeros in $D_\rho(x)$, counted with multiplicity.

\begin{corollary}[Local Hankel-positivity formulation of RH]
\label{cor:rh-local-hankel-positivity}
The Riemann hypothesis is equivalent to the following statement: for
every $x\in\mathbb R$ and every $\rho>0$ such that $\Xi$ has no
zero on $\Gamma_\rho(x)$,
\begin{equation}
H_m\bigl(\Gamma_\rho(x);\phi_{x,\rho}\bigr)\succeq0
\qquad
\text{for every }m\geq1.
\label{eq:rh-all-local-hankel-psd}
\end{equation}
Equivalently, whenever $\mathcal N_\rho(x)>0$, it is sufficient to require
\begin{equation}
H_{\mathcal N_\rho(x)}
\bigl(\Gamma_\rho(x);\phi_{x,\rho}\bigr)
\succeq0.
\label{eq:rh-finite-local-psd}
\end{equation}
\end{corollary}

\begin{proof}
If RH holds, then all zeros enclosed by a real-centered disk are real.
Proposition~\ref{prop:psd-real-zeros} therefore gives
\eqref{eq:rh-all-local-hankel-psd}.

Conversely, suppose that $\Xi$ has a nonreal zero $\tau$. By
conjugation symmetry, $\overline\tau$ is also a zero with the same
multiplicity. Set $x=\operatorname{Re}\tau$, and choose
$\rho>|\operatorname{Im}\tau|$ so that no zero lies on
$\Gamma_\rho(x)$. The disk contains the conjugate pair
$\tau,\overline\tau$. If $r$ is the number of distinct coordinate
nodes in the disk, then
Theorem~\ref{thm:inertia-conjugation-symmetric} shows that
$H_m$ has a negative eigenvalue for every $m\geq r$, contradicting
\eqref{eq:rh-all-local-hankel-psd}.

Finally, $r\leq \mathcal N_\rho(x)$, because the latter counts multiplicity.
Hence \eqref{eq:rh-finite-local-psd} is sufficient to invoke
Corollary~\ref{cor:local-real-support} in each nonempty disk.
\end{proof}

\begin{remark}
Corollary~\ref{cor:rh-local-hankel-positivity} is a local
finite-dimensional reformulation of RH, not a proof. Establishing the
positivity in \eqref{eq:rh-finite-local-psd} directly from the analytic
or arithmetic structure of $\Xi$, without prior information on its
zeros, remains the essential unresolved step.
\end{remark}

\subsection{Dependence on coordinates and contour deformation}

The contour determines which zeros contribute to the moment sequence,
whereas the coordinate function determines their numerical
representation. We first record covariance under affine coordinate
changes.

Let
\begin{equation}
\psi_\Gamma=a\phi_\Gamma+b,
\qquad
a,b\in\mathbb C,
\quad
a\neq0,
\label{eq:affine-coordinate-change}
\end{equation}
and define $T_m(a,b)\in\mathbb C^{m\times m}$ by
\begin{equation}
T_m(a,b)_{kr}
=
\binom{k}{r}a^r b^{k-r},
\qquad
0\leq r\leq k\leq m-1,
\label{eq:T-affine-entries}
\end{equation}
with all remaining entries equal to zero. Then
$$
v_m(au+b)=T_m(a,b)v_m(u).
$$

\begin{proposition}[Affine coordinate covariance]
\label{prop:affine-coordinate-covariance}
Under the affine change \eqref{eq:affine-coordinate-change},
\begin{equation}
H_m(\Gamma;\psi_\Gamma)
=
T_m(a,b)
H_m(\Gamma;\phi_\Gamma)
T_m(a,b)^T.
\label{eq:affine-congruence}
\end{equation}
Consequently,
\begin{equation}
\operatorname{rank}
H_m(\Gamma;\psi_\Gamma)
=
\operatorname{rank}
H_m(\Gamma;\phi_\Gamma).
\label{eq:affine-rank-invariance}
\end{equation}
If $a,b\in\mathbb R$ and the matrices are real symmetric, then they
have the same inertia. In particular, positive semidefiniteness and
the presence of negative eigenvalues are preserved.
\end{proposition}

\begin{proof}
Let $u_\ell=\phi_\Gamma(\tau_\ell)$. Then
$$
\psi_\Gamma(\tau_\ell)=au_\ell+b
$$
and
$$
v_m\bigl(\psi_\Gamma(\tau_\ell)\bigr)
=
T_m(a,b)v_m(u_\ell).
$$
The Vandermonde factorization gives
\begin{align*}
H_m(\Gamma;\psi_\Gamma)
&=
\sum_\ell n_\ell
T_m(a,b)v_m(u_\ell)v_m(u_\ell)^TT_m(a,b)^T
\\
&=
T_m(a,b)H_m(\Gamma;\phi_\Gamma)T_m(a,b)^T.
\end{align*}
The matrix $T_m(a,b)$ is triangular with diagonal entries
$1,a,\ldots,a^{m-1}$, and is therefore nonsingular. Rank invariance
follows. For real $a$ and $b$, inertia preservation follows from
Sylvester's law of inertia.
\end{proof}

For circular contours centered at $x\in\mathbb R$, the raw coordinate
$t$ and the normalized centered coordinate
$$
\phi_{x,\rho}(t)=\frac{t-x}{\rho}
$$
are related by a real affine transformation. Centering and scaling do
not create the exact rank, inertia, or positivity properties. Their
purpose is numerical: the normalized coordinate keeps enclosed nodes
inside the unit disk, controls high powers, and improves the
conditioning of the resulting moment and Hankel matrices.

We next consider deformation of the contour-coordinate pair. Let
$I\subset\mathbb R$ be an interval, and let $\Gamma_s$, $s\in I$,
be a deformation of positively oriented Jordan contours. Denote by
$\Omega_s$ the enclosed domain, and write
$$
\phi_s:=\phi_{\Gamma_s}.
$$
We assume that the coordinate values of the zeros under consideration
depend continuously on $s$.

\begin{proposition}[Rank stability under contour deformation]
\label{prop:rank-invariance-moving-contours}
Suppose that no zero of $\Xi$ crosses $\Gamma_s$ for
$s\in[s_0,s_1]$, so that the same zeros remain enclosed throughout
the deformation. Assume also that their distinct coordinate images
remain distinct. If their number is $r$, then
\begin{equation}
\operatorname{rank}
H_m(\Gamma_s;\phi_s)=r,
\qquad
m\geq r,
\quad
s\in[s_0,s_1].
\label{eq:rank-stability}
\end{equation}
\end{proposition}

\begin{proof}
For every fixed $s$, the reduced coordinate measure has exactly
$r$ distinct support points. Proposition~\ref{prop:rank-property}
therefore applies at each parameter value.
\end{proof}

The proposition concerns rank stability rather than entrywise
constancy. Even when the enclosed zero set is unchanged, the entries of
$H_m(\Gamma_s;\phi_s)$ generally vary because the coordinate
function changes with $s$. When a zero crosses the moving contour,
the coordinate measure gains or loses a weighted atom. The resulting
finite-rank jump and its dynamical consequences are developed in the
next section.

\section{Hankel dynamics generated by moving contours}
\label{sec:hankel-dynamics}

The preceding section describes the static information carried by a
contour Hankel matrix. We now let the contour-coordinate pair vary and
interpret the resulting matrices as a trajectory in a finite-dimensional
matrix space. Each enclosed zero contributes a rank-one feature
$$
n_\ell
v_m\bigl(u_\ell(s)\bigr)
v_m\bigl(u_\ell(s)\bigr)^T,
$$
so the evolution has two distinct components. As long as the enclosed
zero set is unchanged, only the coordinate representation moves and
the matrix varies continuously. When a zero crosses the contour, the
atomic measure gains or loses one or more weighted atoms and the matrix
undergoes a finite-rank jump. The purpose of this section is to derive
these two components, combine them into a hybrid matrix dynamical
system, and show how the crossing data can be recovered from its jumps.

Let
$$
s\longmapsto(\Gamma_s,\phi_s),
\qquad
s\in I,
$$
be a one-parameter family of contour-coordinate pairs, and let
$\Omega_s$ denote the bounded domain enclosed by $\Gamma_s$. We
assume that $\phi_s(t)$ depends continuously on $s$, is
continuously differentiable away from isolated crossing parameters,
and is holomorphic in a neighborhood of
$\overline{\Omega_s}$. For every admissible $s$, define
\begin{equation}
H_m(s)
=
H_m(\Gamma_s;\phi_s).
\label{eq:hankel-flow-definition}
\end{equation}
If $\tau_\ell$ is a zero of $\Xi$ of multiplicity $n_\ell$, set
$$
\chi_\ell(s)
=
\begin{cases}
1, & \tau_\ell\in\Omega_s,\\
0, & \tau_\ell\notin\overline{\Omega_s},
\end{cases}
\qquad
u_\ell(s)=\phi_s(\tau_\ell).
$$
Away from crossing parameters,
\begin{equation}
H_m(s)
=
\sum_\ell
n_\ell\chi_\ell(s)
v_m\bigl(u_\ell(s)\bigr)
v_m\bigl(u_\ell(s)\bigr)^T.
\label{eq:hankel-flow-atomic}
\end{equation}

\subsection{Continuous flow between crossing events}

We first consider an interval on which no zero crosses the moving
contour. In that case the indicators $\chi_\ell(s)$ are constant,
and the evolution is generated entirely by the motion of the
coordinate values $u_\ell(s)$.

Let
$$
d_m(u)
:=
\frac{\mathrm d}{\mathrm du}v_m(u)
=
\bigl(0,1,2u,\ldots,(m-1)u^{m-2}\bigr)^T.
$$

\begin{proposition}[Continuous Hankel evolution]
\label{prop:continuous-hankel-evolution}
Suppose that no zero of $\Xi$ crosses $\Gamma_s$ for
$s\in(s_0,s_1)$. Then $H_m(s)$ is continuously differentiable on
$(s_0,s_1)$ and satisfies
\begin{equation}
\frac{\mathrm d}{\mathrm ds}H_m(s)
=
\sum_{\tau_\ell\in\Omega_s}
n_\ell\dot u_\ell(s)
\left[
d_m\bigl(u_\ell(s)\bigr)v_m\bigl(u_\ell(s)\bigr)^T
+
v_m\bigl(u_\ell(s)\bigr)d_m\bigl(u_\ell(s)\bigr)^T
\right].
\label{eq:general-continuous-hankel-flow}
\end{equation}
\end{proposition}

\begin{proof}
On a crossing-free interval, the same zeros remain enclosed.
Differentiating each rank-one term in
\eqref{eq:hankel-flow-atomic} gives
$$
\frac{\mathrm d}{\mathrm ds}
\left[
v_m(u_\ell)v_m(u_\ell)^T
\right]
=
\dot u_\ell
\left[
d_m(u_\ell)v_m(u_\ell)^T
+
v_m(u_\ell)d_m(u_\ell)^T
\right].
$$
Summing over the enclosed zeros proves the result.
\end{proof}

For a general contour-coordinate family,
\eqref{eq:general-continuous-hankel-flow} is already a closed
description in terms of the moving coordinate nodes. For circular
contours, however, the coordinate motion is affine and the equation
reduces to a finite-dimensional Lyapunov-type system.

Consider
\begin{equation}
\Gamma_s
=
\Gamma_{\rho(s)}\bigl(z(s)\bigr),
\qquad
\phi_s(t)
=
\frac{t-z(s)}{\rho(s)},
\label{eq:moving-circular-family}
\end{equation}
where $z(s)\in\mathbb C$, $\rho(s)>0$, and both functions are
continuously differentiable. For every enclosed zero,
$$
u_\ell(s)
=
\frac{\tau_\ell-z(s)}{\rho(s)},
$$
and hence
\begin{equation}
\dot u_\ell(s)
=
-\frac{\dot z(s)}{\rho(s)}
-
\frac{\dot\rho(s)}{\rho(s)}u_\ell(s).
\label{eq:coordinate-node-dynamics}
\end{equation}
Introduce
$$
D_m
=
\begin{pmatrix}
0      &        &        &        & 0\\
1      & 0      &        &        &  \\
0      & 2      & 0      &        &  \\
\vdots & \ddots & \ddots & \ddots &  \\
0      & \cdots & 0      & m-1    & 0
\end{pmatrix},
\qquad
N_m
=
\operatorname{diag}(0,1,\ldots,m-1).
$$
Then
$$
D_mv_m(u)=d_m(u),
\qquad
N_mv_m(u)=u\,d_m(u).
$$
Set
\begin{equation}
\alpha(s)
=
-\frac{\dot z(s)}{\rho(s)},
\qquad
\beta(s)
=
-\frac{\dot\rho(s)}{\rho(s)},
\qquad
A_m(s)
=
\alpha(s)D_m+\beta(s)N_m.
\label{eq:hankel-generator}
\end{equation}
Equation \eqref{eq:coordinate-node-dynamics} then gives
$$
\frac{\mathrm d}{\mathrm ds}
v_m\bigl(u_\ell(s)\bigr)
=
A_m(s)v_m\bigl(u_\ell(s)\bigr).
$$

\begin{theorem}[Continuous circular Hankel flow]
\label{thm:circular-hankel-flow}
Between contour-crossing parameters, the Hankel matrix associated with
\eqref{eq:moving-circular-family} satisfies
\begin{equation}
\frac{\mathrm d}{\mathrm ds}H_m(s)
=
A_m(s)H_m(s)+H_m(s)A_m(s)^T.
\label{eq:lyapunov-hankel-flow}
\end{equation}
\end{theorem}

\begin{proof}
Using the atomic representation and
$$
\frac{\mathrm d}{\mathrm ds}v_m(u_\ell)
=
A_m v_m(u_\ell),
$$
we obtain
\begin{align*}
\dot H_m(s)
&=
\sum_{\tau_\ell\in\Omega_s}n_\ell
\left[
A_m(s)v_m(u_\ell)v_m(u_\ell)^T
+
v_m(u_\ell)v_m(u_\ell)^TA_m(s)^T
\right]
\\
&=
A_m(s)H_m(s)+H_m(s)A_m(s)^T.
\end{align*}
\end{proof}

Let $\Phi_m(s,s_0)$ be the fundamental matrix defined by
\begin{equation}
\frac{\mathrm d}{\mathrm ds}\Phi_m(s,s_0)
=
A_m(s)\Phi_m(s,s_0),
\qquad
\Phi_m(s_0,s_0)=I.
\label{eq:fundamental-matrix}
\end{equation}
Then, on every crossing-free interval,
\begin{equation}
H_m(s)
=
\Phi_m(s,s_0)
H_m(s_0)
\Phi_m(s,s_0)^T.
\label{eq:hankel-flow-congruence}
\end{equation}
Thus the continuous motion acts by congruence. In particular, rank is
preserved between crossings, and inertia is preserved whenever the
flow is real and the Hankel matrices are real symmetric.

For a fixed radius $\rho(s)\equiv\rho$,
$$
A_m(s)
=
-\frac{\dot z(s)}{\rho}D_m,
$$
so
\begin{equation}
\dot H_m(s)
=
-\frac{\dot z(s)}{\rho}
\left[
D_mH_m(s)+H_m(s)D_m^T
\right],
\label{eq:fixed-radius-hankel-flow}
\end{equation}
and, since $D_m$ is nilpotent,
\begin{equation}
\Phi_m(s,s_0)
=
\exp\left(
-\frac{z(s)-z(s_0)}{\rho}D_m
\right).
\label{eq:translation-flow-map}
\end{equation}
The same evolution can be expressed at the moment level as
\begin{equation}
\dot\mu_k(s)
=
k\alpha(s)\mu_{k-1}(s)
+
k\beta(s)\mu_k(s),
\qquad
k\geq1,
\label{eq:moment-dynamical-system}
\end{equation}
together with
\begin{equation}
\dot\mu_0(s)=0.
\label{eq:zeroth-moment-conservation}
\end{equation}
Hence the zeroth moment, and therefore the enclosed zero count, is
constant until a crossing occurs.

\subsection{Crossing impulses and the hybrid Hankel system}

We now turn to the discontinuous part of the evolution. The continuous
flow describes a change of coordinates for a fixed enclosed zero set;
a crossing changes the zero set itself. At the level of the atomic
measure, this means adding or removing a weighted atom, and at the
matrix level it produces a finite-rank impulse.

Let $s_*$ be an isolated crossing parameter, and assume first that
exactly one zero $\tau_*$ crosses $\Gamma_s$ at $s_*$. Define
$$
H_m(s_*^\pm)
:=
\lim_{s\to s_*^\pm}H_m(s),
\qquad
\mu_k(s_*^\pm)
:=
\lim_{s\to s_*^\pm}\mu_k(\Gamma_s;\phi_s),
$$
and assume that the coordinate of the crossing zero has the common
one-sided limit
$$
u_*
:=
\lim_{s\to s_*^-}\phi_s(\tau_*)
=
\lim_{s\to s_*^+}\phi_s(\tau_*).
$$
Let
$$
\varepsilon_*
=
\begin{cases}
+1, & \tau_*\text{ enters }\Omega_s,\\
-1, & \tau_*\text{ leaves }\Omega_s.
\end{cases}
$$

\begin{theorem}[Rank-one contour-crossing law]
\label{thm:rank-one-contour-crossing}
If $\tau_*$ has multiplicity $n_*$, then
\begin{equation}
\mu_k(s_*^+)-\mu_k(s_*^-)
=
\varepsilon_*n_*u_*^k,
\qquad
k\geq0,
\label{eq:moment-jump-crossing}
\end{equation}
and
\begin{equation}
H_m(s_*^+)-H_m(s_*^-)
=
\varepsilon_*n_*
v_m(u_*)v_m(u_*)^T.
\label{eq:rank-one-hankel-impulse}
\end{equation}
\end{theorem}

\begin{proof}
The contributions of all zeros that remain enclosed have the same
one-sided limits and therefore cancel in the difference. The atomic
measure gains or loses the single weighted atom
$n_*\delta_{u_*}$. The residue representation yields
\eqref{eq:moment-jump-crossing}, and the Vandermonde factorization
then gives \eqref{eq:rank-one-hankel-impulse}.
\end{proof}

Thus an isolated crossing is represented by a signed rank-one matrix.
Its scalar weight records the multiplicity and direction of passage,
while its matrix direction is determined by the crossing coordinate.
For real-axis scans, this rank-one law also separates real crossings
from nonreal conjugate crossings.

\begin{proposition}[Classification of isolated symmetric crossing events]
\label{prop:symmetric-crossing-classification}
Consider fixed-radius circular contours
$$
\Gamma_x=\Gamma_\rho(x),
\qquad
x\in\mathbb R,
$$
equipped with
$$
\phi_x(t)=\frac{t-x}{\rho}.
$$
Assume that a crossing parameter $x_*$ is isolated from all other
conjugacy orbits of zeros.

If a real zero $\gamma$ of multiplicity $n$ crosses the contour,
then
$$
u_*
=
\frac{\gamma-x_*}{\rho}
\in\{-1,1\},
$$
and the jump is a rank-one signed semidefinite matrix.

If a nonreal conjugate pair
$\tau,\overline\tau$, of common multiplicity $n$, crosses the
contour, then both zeros cross at the same parameter. For every
$m\geq2$, the jump has rank two and inertia
\begin{equation}
\operatorname{In}
\bigl(H_m(x_*^+)-H_m(x_*^-)\bigr)
=
(1,1,m-2).
\label{eq:nonreal-crossing-inertia}
\end{equation}
\end{proposition}

\begin{proof}
For a real crossing, the crossing coordinate is real and has modulus
one. Theorem~\ref{thm:rank-one-contour-crossing} therefore gives a
signed rank-one semidefinite matrix.

For a nonreal zero and a real center,
$$
|\tau-x|=|\overline\tau-x|,
$$
so the two members of the conjugate pair cross simultaneously. Their
coordinates are $u_*$ and $\overline u_*$, with
$\operatorname{Im}u_*\neq0$. Up to the common crossing sign, the
jump is the contour Hankel matrix associated with a single nonreal
conjugate pair. Theorem~\ref{thm:inertia-conjugation-symmetric}, with
$r_0=0$, $c=1$, and $r=2$, gives
\eqref{eq:nonreal-crossing-inertia}. Multiplication by $-1$
interchanges the positive and negative subspaces, but the inertia
remains $(1,1,m-2)$.
\end{proof}

\begin{remark}[Dynamical interpretation of RH]
Fix $\rho>0$ and consider a generic real-axis scan by circles of radius $\rho$, with distinct conjugacy orbits crossing separately. The absence of rank-two indefinite events is then equivalent to the absence of nonreal zeros in the horizontal strip
$$
|\operatorname{Im}t|<\rho.
$$
Consequently, RH is equivalent to the absence of rank-two indefinite crossing events for every radius $\rho>0$. Thus the static negative index in Theorem~\ref{thm:inertia-conjugation-symmetric} has a dynamical counterpart: real zeros generate rank-one signed semidefinite events, whereas nonreal conjugate pairs generate rank-two indefinite events whenever the scanning radius reaches them. This is a reformulation, rather than a proof, of RH.
\end{remark}

The single-crossing law extends by superposition. Let
$$
\sigma_1<\sigma_2<\cdots
$$
be the distinct crossing parameters. At $\sigma_j$, suppose that
$p_j$ distinct coordinate nodes cross, and define
\begin{equation}
Q_j
=
\sum_{r=1}^{p_j}
c_{j,r}
v_m(u_{j,r})v_m(u_{j,r})^T,
\qquad
c_{j,r}
=
\varepsilon_{j,r}n_{j,r}\neq0.
\label{eq:crossing-input-matrix}
\end{equation}

\begin{theorem}[Impulsive Hankel dynamical system]
\label{thm:impulsive-hankel-system}
For the circular family \eqref{eq:moving-circular-family}, the contour
Hankel trajectory satisfies, in the sense of matrix-valued
distributions,
\begin{equation}
\frac{\mathrm d}{\mathrm ds}H_m(s)
=
A_m(s)H_m(s)+H_m(s)A_m(s)^T
+
\sum_jQ_j\delta(s-\sigma_j).
\label{eq:impulsive-hankel-dynamics}
\end{equation}
Equivalently, $H_m$ satisfies the continuous equation between event
parameters and the jump conditions
$$
H_m(\sigma_j^+)
=
H_m(\sigma_j^-)+Q_j.
$$
\end{theorem}

This equation combines the two mechanisms derived above. The
continuous congruence flow changes only the coordinate representation,
whereas all changes in the enclosed support occur through the
finite-rank matrices $Q_j$. In particular, the zeroth moment is
constant between events and jumps by the signed total multiplicity at
each event.

\subsection{Drift correction and recovery of crossing events}
\label{subsec:hankel-identification}

The impulsive system separates the dynamics conceptually, but the raw
trajectory still contains both coordinate drift and crossing jumps. We
therefore first remove the known continuous flow. The corrected
trajectory is piecewise constant, so its discontinuities isolate the
event matrices and provide the appropriate data for recovery.

Let $\Phi_m(s,r)$ be the transition matrix generated by $A_m$ and introduced in \eqref{eq:fundamental-matrix}, now written with an arbitrary initial parameter $r$. It satisfies
\begin{align}
\Phi_m(s,r)\Phi_m(r,t)&=\Phi_m(s,t),\\
\Phi_m(s,r)^{-1}&=\Phi_m(r,s).
\end{align}

\begin{theorem}[Event decomposition of the Hankel flow]
\label{thm:hankel-event-decomposition}
Let $s\geq s_0$, and assume that only finitely many crossing events
occur in $[s_0,s]$. Then
\begin{equation}
\begin{split}
H_m(s)
={}&
\Phi_m(s,s_0)H_m(s_0)\Phi_m(s,s_0)^T
\\
&+
\sum_{\sigma_j\in(s_0,s]}
\Phi_m(s,\sigma_j)Q_j\Phi_m(s,\sigma_j)^T.
\end{split}
\label{eq:hankel-event-decomposition}
\end{equation}
\end{theorem}

\begin{proof}
Between consecutive crossing parameters, the state is transported by
the continuous congruence flow. At $\sigma_j$, it acquires the
increment $Q_j$. Propagating each increment from $\sigma_j$ to
$s$ yields \eqref{eq:hankel-event-decomposition}.
\end{proof}

Fix $s_0$ and define
\begin{equation}
K_m(s)
=
\Phi_m(s,s_0)^{-1}
H_m(s)
\Phi_m(s,s_0)^{-T},
\label{eq:drift-corrected-hankel-state}
\end{equation}
where $X^{-T}=(X^{-1})^T$.

\begin{corollary}[Piecewise-constant drift-corrected state]
\label{cor:piecewise-constant-corrected-state}
The matrix $K_m(s)$ is constant on every crossing-free interval. At
an event parameter $\sigma_j$,
\begin{equation}
K_m(\sigma_j^+)-K_m(\sigma_j^-)
=
\Phi_m(\sigma_j,s_0)^{-1}
Q_j
\Phi_m(\sigma_j,s_0)^{-T}.
\label{eq:corrected-state-jump}
\end{equation}
In particular, rank is preserved by the drift correction. In a real
symmetric flow, inertia is preserved as well.
\end{corollary}

\begin{proof}
Differentiating \eqref{eq:drift-corrected-hankel-state} away from the
crossing parameters and using
$$
\dot H_m=A_mH_m+H_mA_m^T
$$
gives $\dot K_m=0$. The jump formula follows by applying the same
invertible congruence transformation to $Q_j$.
\end{proof}

Hence all crossing information is concentrated in the jumps of
$K_m$. We first consider an isolated single-zero event and write
\begin{equation}
\Delta H_m(s_*)
=
H_m(s_*^+)-H_m(s_*^-).
\label{eq:hankel-jump-definition}
\end{equation}

\begin{theorem}[Exact recovery from an isolated crossing]
\label{thm:single-crossing-recovery}
Suppose that exactly one zero $\tau_*$ crosses the contour at
$s=s_*$, with multiplicity $n_*$, crossing coordinate
$$
u_*=\phi_{s_*}(\tau_*),
$$
and direction $\varepsilon_*\in\{-1,1\}$. If $m\geq2$, then
\begin{equation}
\varepsilon_*n_*
=
\bigl(\Delta H_m(s_*)\bigr)_{00},
\label{eq:recover-signed-multiplicity}
\end{equation}
and
\begin{equation}
u_*
=
\frac{\bigl(\Delta H_m(s_*)\bigr)_{10}}
     {\bigl(\Delta H_m(s_*)\bigr)_{00}}.
\label{eq:recover-crossing-coordinate}
\end{equation}
Consequently,
\begin{equation}
n_*
=
\left|
\bigl(\Delta H_m(s_*)\bigr)_{00}
\right|,
\qquad
\varepsilon_*
=
\operatorname{sign}
\bigl(\Delta H_m(s_*)\bigr)_{00}.
\label{eq:recover-multiplicity-direction}
\end{equation}
\end{theorem}

\begin{proof}
The rank-one crossing law gives
$$
\Delta H_m(s_*)
=
\varepsilon_*n_*
v_m(u_*)v_m(u_*)^T.
$$
Since the first two entries of $v_m(u_*)$ are $1$ and $u_*$, the
$(0,0)$- and $(1,0)$-entries give the stated formulas.
\end{proof}

For a circular family, the recovered coordinate determines the zero
location directly.

\begin{corollary}[Recovery of a crossing zero]
\label{cor:circular-crossing-recovery}
Suppose that
$$
\Gamma_s
=
\Gamma_{\rho(s)}\bigl(z(s)\bigr),
\qquad
\phi_s(t)
=
\frac{t-z(s)}{\rho(s)}.
$$
Under the assumptions of
Theorem~\ref{thm:single-crossing-recovery},
\begin{equation}
\tau_*
=
z(s_*)+\rho(s_*)u_*.
\label{eq:recover-crossing-zero}
\end{equation}
At an exact crossing, $|u_*|=1$, which provides an internal
consistency check.
\end{corollary}

The same formulas remain stable under a small perturbation of the
event matrix.

\begin{proposition}[Stability of isolated-event recovery]
\label{prop:single-crossing-stability}
Let
$$
Q
=
\varepsilon n\,v_m(u)v_m(u)^T,
\qquad
m\geq2,
$$
be an exact single-event matrix, and suppose that the observed jump is
$$
\widehat Q=Q+E.
$$
Set $\eta=\|E\|_F$, assume $\eta<n$, and define
\begin{equation}
\widehat u
=
\frac{\widehat Q_{10}}{\widehat Q_{00}}.
\label{eq:perturbed-coordinate-recovery}
\end{equation}
Then
\begin{equation}
|\widehat u-u|
\leq
\frac{(1+|u|)\eta}{n-\eta}.
\label{eq:coordinate-stability-bound}
\end{equation}
For a circular crossing, where $|u|=1$,
\begin{equation}
|\widehat u-u|
\leq
\frac{2\eta}{n-\eta}.
\label{eq:circular-coordinate-stability}
\end{equation}
If $\eta\leq n/2$, then
\begin{equation}
|\widehat u-u|
\leq
\frac{4\eta}{n}.
\label{eq:simplified-coordinate-stability}
\end{equation}
Moreover, if $|E_{00}|<1/2$, then the integer multiplicity is
recovered exactly by
\begin{equation}
\widehat n
=
\operatorname{round}\bigl(|\widehat Q_{00}|\bigr).
\label{eq:multiplicity-rounding}
\end{equation}
\end{proposition}

\begin{proof}
Since
$$
Q_{00}=\varepsilon n,
\qquad
Q_{10}=\varepsilon n u,
$$
we have
$$
\widehat u-u
=
\frac{E_{10}-uE_{00}}{Q_{00}+E_{00}}.
$$
Using $|E_{00}|\leq\eta$ and $|E_{10}|\leq\eta$ gives
$$
|\widehat u-u|
\leq
\frac{|E_{10}|+|u||E_{00}|}{n-|E_{00}|}
\leq
\frac{(1+|u|)\eta}{n-\eta}.
$$
The remaining estimates follow from $|u|=1$ and
$\eta\leq n/2$. The rounding statement follows from
$$
|\widehat Q_{00}-\varepsilon n|<\frac12.
$$
\end{proof}

For a circular event, the reconstructed location
$$
\widehat\tau
=
z(s_*)+\rho(s_*)\widehat u
$$
therefore satisfies
\begin{equation}
|\widehat\tau-\tau_*|
\leq
\rho(s_*)
\frac{2\eta}{n-\eta}.
\label{eq:zero-location-stability}
\end{equation}

The single-event formula uses only two entries of the jump matrix. When
several distinct coordinate nodes cross simultaneously, the full jump
moment sequence is required, and the event can be recovered by a
finite Prony-type reconstruction.

\begin{theorem}[Recovery of simultaneous crossing events]
\label{thm:multiple-crossing-recovery}
Suppose that $p$ distinct coordinate nodes
$$
u_1,\ldots,u_p
$$
cross at the same parameter, with nonzero signed multiplicities
$$
c_j=\varepsilon_jn_j,
\qquad
j=1,\ldots,p.
$$
Assume that coincident coordinate values have been aggregated and that
no aggregated signed weight vanishes. Define
\begin{equation}
y_k
:=
\mu_k^+-\mu_k^-
=
\sum_{j=1}^{p}c_ju_j^k.
\label{eq:jump-moment-sequence}
\end{equation}
If $m\geq p+1$, then the jump matrix
$$
\Delta H_m
=
\bigl(y_{r+s}\bigr)_{r,s=0}^{m-1}
$$
uniquely determines $p$, the nodes
$u_1,\ldots,u_p$, and the signed multiplicities
$c_1,\ldots,c_p$, up to permutation.
\end{theorem}

\begin{proof}
The jump matrix admits the factorization
$$
\Delta H_m
=
V_m\operatorname{diag}(c_1,\ldots,c_p)V_m^T.
$$
Its rank is at most $p$. On the other hand, its leading
$p\times p$ principal block is
$$
\Delta H_p
=
V_p\operatorname{diag}(c_1,\ldots,c_p)V_p^T,
$$
where $V_p$ is a square Vandermonde matrix. Since the nodes are
distinct and the weights are nonzero, both $V_p$ and
$\operatorname{diag}(c_1,\ldots,c_p)$ are nonsingular. Hence
$\Delta H_p$ is nonsingular, and therefore
$$
\operatorname{rank}\Delta H_m=p.
$$
Thus $p$ is determined by the rank.

Let
$$
\pi(\lambda)
=
\prod_{j=1}^{p}(\lambda-u_j)
=
\lambda^p+a_{p-1}\lambda^{p-1}+\cdots+a_0.
$$
The relations $\pi(u_j)=0$ imply
$$
y_{k+p}
+
\sum_{r=0}^{p-1}a_ry_{k+r}
=
0,
\qquad
k\geq0.
$$
For $k=0,\ldots,p-1$, this gives
\begin{equation}
\begin{pmatrix}
y_0&y_1&\cdots&y_{p-1}\\
y_1&y_2&\cdots&y_p\\
\vdots&\vdots&\ddots&\vdots\\
y_{p-1}&y_p&\cdots&y_{2p-2}
\end{pmatrix}
\begin{pmatrix}
a_0\\a_1\\\vdots\\a_{p-1}
\end{pmatrix}
=
-
\begin{pmatrix}
y_p\\y_{p+1}\\\vdots\\y_{2p-1}
\end{pmatrix}.
\label{eq:annihilating-polynomial-system}
\end{equation}
The coefficient matrix is precisely the nonsingular block
$\Delta H_p$. Hence the coefficients of $\pi$ are uniquely
determined, and its roots give the nodes $u_1,\ldots,u_p$.

Finally, the signed multiplicities are recovered from
\begin{equation}
\begin{pmatrix}
1&\cdots&1\\
u_1&\cdots&u_p\\
\vdots&\ddots&\vdots\\
u_1^{p-1}&\cdots&u_p^{p-1}
\end{pmatrix}
\begin{pmatrix}
c_1\\\vdots\\c_p
\end{pmatrix}
=
\begin{pmatrix}
y_0\\\vdots\\y_{p-1}
\end{pmatrix}.
\label{eq:recover-crossing-weights}
\end{equation}
The Vandermonde matrix is invertible because the nodes are distinct.
\end{proof}

\begin{remark}[Conditioning of simultaneous-event recovery]
The conditioning of the reconstruction in Theorem~\ref{thm:multiple-crossing-recovery} is governed by the Hankel, Vandermonde, and annihilating-polynomial systems appearing in the proof. It deteriorates when coordinate nodes approach one another or when a nonzero aggregated signed weight approaches zero. At an exact coordinate collision, the corresponding atoms are no longer separately identifiable from the jump moments.
\end{remark}

The results of this section give a unified interpretation of the
moving-contour Hankel matrix. Between crossings, the local zero
configuration is transported through a continuous congruence flow. At
a crossing, the addition or removal of zero features produces a
finite-rank impulse. After the known drift is removed, these impulses
form an identifiable event process: real and nonreal crossing orbits
have different inertia signatures, and isolated events determine the
direction, multiplicity, coordinate, and, for circular contours, the
location of the crossing zero.

\section{Indicator fields generated by moving contours}
\label{sec:indicator-functionals}

The Hankel dynamics provide a matrix-valued description of the zero
configuration observed by a moving contour. We now introduce scalar
observables that record the zero count, the magnitude of the Hankel
state, and the failure of its real positive-semidefinite signature.

For a contour-coordinate pair $(\Gamma,\phi_\Gamma)$, define the
zero-count observable
\begin{equation}
\mathcal N(\Gamma)
=
\frac{1}{2\pi\mathrm i}
\int_\Gamma G(t)\,\mathrm dt
=
\mu_0(\Gamma;\phi_\Gamma).
\label{eq:contour-zero-count}
\end{equation}
By the argument principle,
$$
\mathcal N(\Gamma)=N_{\Omega_\Gamma}(\Xi)
$$
is the number of zeros in $\Omega_\Gamma$, counted with
multiplicity, and is independent of $\phi_\Gamma$. Its binary
version is
\begin{equation}
\mathcal P(\Gamma)
=
\min\{1,\mathcal N(\Gamma)\}.
\label{eq:contour-count-indicator}
\end{equation}
Thus $\mathcal P(\Gamma)=0$ precisely for a zero-free contour
interior.

The Hankel-energy observable is
\begin{equation}
\mathcal E_m(\Gamma;\phi_\Gamma)
=
\log\left(1+\|H_m(\Gamma;\phi_\Gamma)\|_F\right).
\label{eq:hankel-energy}
\end{equation}
Unlike $\mathcal N$, this quantity depends on the coordinate because
it retains the locations and multiplicities of all enclosed atoms.
For circular contours we therefore use the centered normalized
coordinate $\phi_{z,\rho}(t)=(t-z)/\rho$.

When $H_m(\Gamma;\phi_\Gamma)$ is real symmetric, define the
positive-semidefiniteness violation observable
\begin{equation}
\mathcal V_m(\Gamma;\phi_\Gamma)
=
\max\left\{
0,
-\lambda_{\min}\bigl(H_m(\Gamma;\phi_\Gamma)\bigr)
\right\}.
\label{eq:psd-violation}
\end{equation}
This definition is not extended to a general complex symmetric matrix
by taking a Hermitian part, since that would change the matrix object
studied in the preceding sections.

\begin{proposition}[Evolution of the indicator observables]
\label{prop:indicator-evolution}
Let $s\mapsto(\Gamma_s,\phi_s)$ be a moving contour-coordinate pair
with isolated crossing parameters $s_j$. Write
$$
\mathcal N(s)=\mathcal N(\Gamma_s),
\qquad
\mathcal E_m(s)=\mathcal E_m(\Gamma_s;\phi_s).
$$
Between crossings,
\begin{equation}
\frac{\mathrm d}{\mathrm ds}\mathcal N(s)=0.
\label{eq:count-conservation}
\end{equation}
If a zero of multiplicity $n_j$ enters or leaves at $s_j$, then
\begin{equation}
\mathcal N(s_j^+)-\mathcal N(s_j^-)
=
\varepsilon_j n_j,
\label{eq:count-jump}
\end{equation}
where $\varepsilon_j=1$ for entry and $\varepsilon_j=-1$ for
exit.

Suppose that $H_m(s)\neq0$ and
$$
\dot H_m(s)
=
A_m(s)H_m(s)+H_m(s)A_m(s)^T
$$
between crossings. Then
\begin{equation}
\frac{\mathrm d}{\mathrm ds}\mathcal E_m(s)
=
\frac{
\operatorname{Re}
\left\langle
H_m(s),
A_m(s)H_m(s)+H_m(s)A_m(s)^T
\right\rangle_F
}{
\bigl(1+\|H_m(s)\|_F\bigr)\|H_m(s)\|_F
},
\label{eq:hankel-energy-evolution}
\end{equation}
where
$$
\langle X,Y\rangle_F
=
\operatorname{tr}(X^*Y).
$$
At an isolated crossing with coordinate $u_j$,
\begin{equation}
\mathcal E_m(s_j^+)-\mathcal E_m(s_j^-)
=
\log
\frac{
1+
\left\|
H_m(s_j^-)
+
\varepsilon_j n_jv_m(u_j)v_m(u_j)^T
\right\|_F
}{
1+\|H_m(s_j^-)\|_F
}.
\label{eq:hankel-energy-jump}
\end{equation}
\end{proposition}

\begin{proof}
The count is the zeroth moment and is unchanged while the enclosed zero
set is fixed; its jump is the signed multiplicity of the crossing
zero. For the energy,
$$
\frac{\mathrm d}{\mathrm ds}
\frac12\|H_m(s)\|_F^2
=
\operatorname{Re}\langle H_m(s),\dot H_m(s)\rangle_F.
$$
Combining this identity with \eqref{eq:hankel-energy} gives
\eqref{eq:hankel-energy-evolution}, and the rank-one crossing law gives
\eqref{eq:hankel-energy-jump}.
\end{proof}

The binary field $\mathcal P(\Gamma_s)$ changes only when the zero
count passes between zero and a positive value. For a real symmetric
Hankel flow, inertia is preserved between crossings by the continuous
congruence transformation. Hence
$$
\mathcal V_m(\Gamma_{s_0};\phi_{s_0})=0
\quad\Longleftrightarrow\quad
\mathcal V_m(\Gamma_s;\phi_s)=0
$$
throughout a crossing-free interval, although the numerical value of
$\mathcal V_m$ need not be constant when it is positive.

Let $\mathcal Z(\Xi)$ denote the set of distinct zeros of $\Xi$, and let $n(\tau)$ be the multiplicity of $\tau\in\mathcal Z(\Xi)$. For circular contours, define the center-dependent fields
\begin{equation}
\mathcal N_\rho(z)
=
\mathcal N\bigl(\Gamma_\rho(z)\bigr),
\qquad
\mathcal P_\rho(z)
=
\mathcal P\bigl(\Gamma_\rho(z)\bigr),
\label{eq:circular-count-fields}
\end{equation}
\begin{equation}
\mathcal E_{m,\rho}(z)
=
\mathcal E_m\bigl(\Gamma_\rho(z);\phi_{z,\rho}\bigr),
\label{eq:circular-energy-field}
\end{equation}
and, for real centers producing real symmetric matrices,
\begin{equation}
\mathcal V_{m,\rho}(x)
=
\mathcal V_m\bigl(\Gamma_\rho(x);\phi_{x,\rho}\bigr).
\label{eq:circular-psd-field}
\end{equation}
For every admissible center,
\begin{equation}
\mathcal N_\rho(z)
=
\sum_{\substack{\tau\in\mathcal Z(\Xi)\\|\tau-z|<\rho}}
n(\tau),
\label{eq:circular-local-count}
\end{equation}
and therefore
\begin{equation}
\mathcal P_\rho(z)=1
\quad\Longleftrightarrow\quad
z\in
\bigcup_{\tau\in\mathcal Z(\Xi)}D_\rho(\tau),
\label{eq:count-active-region}
\end{equation}
with the boundary circles excluded from the admissible parameter set.
Thus a nonzero response at the center $z$ means that the disk
centered at $z$ contains a zero; it does not mean that $z$ itself
is a zero.

The energy field provides coordinate-dependent information beyond the zero count because it varies with the normalized coordinates and multiplicities of the enclosed zeros. The
PSD-violation field has a different interpretation. If the number of
distinct coordinate nodes in a real-centered disk is $r$, then
Corollary~\ref{cor:local-real-support} gives
\begin{equation}
\mathcal V_{m,\rho}(x)=0
\quad\Longleftrightarrow\quad
\operatorname{supp}\nu_{\Gamma_\rho(x)}\subset\mathbb R,
\qquad
m\geq r.
\label{eq:psd-field-complete-order}
\end{equation}
For $m<r$, a positive value still certifies a nonreal coordinate
configuration, but a zero value alone is inconclusive.

\section{Numerical implementation and computational validation}
\label{sec:numerics}

The numerical study examines three computational components of the
framework: the accurate evaluation of contour Hankel matrices, the
static information carried by the indicator fields, and the recovery
of a zero from a discretely sampled moving-contour trajectory. The
experiments are limited to representative tests in order to keep the
numerical section proportional to the theoretical development.

\paragraph{Contour quadrature and accuracy.}
For a circular contour
$$
t(\theta)
=
z+\rho e^{\mathrm i\theta},
\qquad
0\leq\theta<2\pi,
$$
the normalized coordinate satisfies
$$
\phi_{z,\rho}(t)
=
\frac{t-z}{\rho}
=
e^{\mathrm i\theta},
\qquad
t\in\Gamma_\rho(z).
$$
Hence
\begin{equation}
\mu_k(z,\rho)
=
\frac{\rho}{2\pi}
\int_0^{2\pi}
e^{\mathrm i(k+1)\theta}
G\bigl(z+\rho e^{\mathrm i\theta}\bigr)
\,\mathrm d\theta,
\label{eq:numerical-circular-moment}
\end{equation}
and the periodic trapezoidal rule gives
\begin{equation}
\mu_k^{(N)}(z,\rho)
=
\frac{\rho}{N}
\sum_{j=0}^{N-1}
e^{\mathrm i(k+1)\theta_j}
G\bigl(z+\rho e^{\mathrm i\theta_j}\bigr),
\qquad
\theta_j=\frac{2\pi j}{N}.
\label{eq:trapezoidal-contour-moment}
\end{equation}

The logarithmic derivative is evaluated without numerical
differentiation. With
$$
w=\frac12+\mathrm i t,
$$
we use
\begin{equation}
\frac{\Xi'(t)}{\Xi(t)}
=
\mathrm i
\left(
\frac1w
+
\frac1{w-1}
-
\frac12\log\pi
+
\frac12\psi\left(\frac{w}{2}\right)
+
\frac{\zeta'(w)}{\zeta(w)}
\right),
\label{eq:xi-log-derivative-evaluation}
\end{equation}
where $\psi$ is the digamma function. The special functions are
evaluated in arbitrary-precision arithmetic, after which the small
Hankel matrices are assembled and analyzed in double precision. For reference quantities obtained from the residue representation, set
$$
\boldsymbol\mu^{(N)}
=
\bigl(\mu_0^{(N)},\ldots,\mu_{2m-2}^{(N)}\bigr)^T,
\qquad
\boldsymbol\mu
=
\bigl(\mu_0,\ldots,\mu_{2m-2}\bigr)^T.
$$
We then define
\begin{equation}
E_\mu
=
\frac{\|\boldsymbol\mu^{(N)}-\boldsymbol\mu\|_2}{\|\boldsymbol\mu\|_2},
\qquad
E_H
=
\frac{\|H_m^{(N)}-H_m\|_F}{\|H_m\|_F}.
\label{eq:numerical-error-measures}
\end{equation}

We first isolate the quadrature error using a finite atomic model with
one simple zero. We take $\rho=0.5$, $m=4$, and place the zero
inside the circle at distances
$$
d=0.20,
\qquad
d=0.10,
\qquad
d=0.05
$$
from the contour. The exact moments are then known explicitly. The
solid curves in \Cref{fig:quadrature-convergence-compact} show that,
for $d=0.20$, the Hankel error reaches the double-precision level at
$N=64$. At the same resolution, the errors for $d=0.10$ and
$d=0.05$ are approximately $6.3\times10^{-7}$ and
$1.2\times10^{-3}$, respectively. Increasing the resolution to
$N=128$ reduces the latter error to approximately
$1.4\times10^{-6}$.

We then repeat the safe-distance test using the first positive zero
$$
\gamma_1\approx14.134725141734695
$$
of $\Xi$. The circle has radius $\rho=0.5$, contains only
$\gamma_1$, and has zero-to-contour distance $d=0.20$. The dashed
curve in \Cref{fig:quadrature-convergence-compact} and the values in
\Cref{tab:xi-quadrature-validation} show that the direct
$\Xi'/\Xi$ computation agrees with the atomic reference and reaches
the double-precision assembly level at $N=64$.

\begin{figure}[htbp]
\centering
\includegraphics[width=0.76\textwidth]
{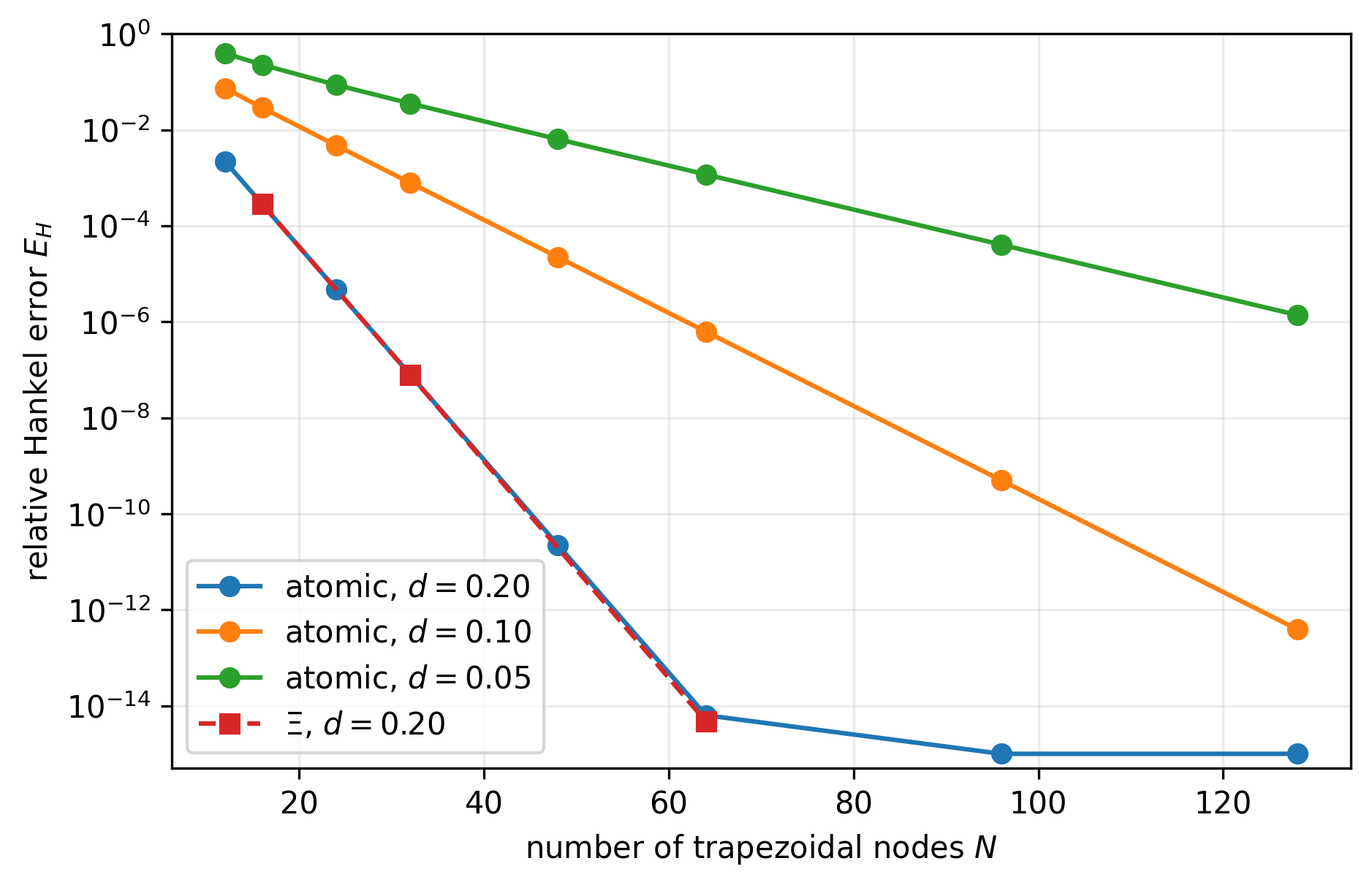}
\caption{Relative Hankel error for circular contour quadrature. The
solid curves correspond to an exact finite atomic model, while the
dashed curve is obtained by direct evaluation of $\Xi'/\Xi$.
Convergence deteriorates as the nearest zero approaches the
integration path.}
\label{fig:quadrature-convergence-compact}
\end{figure}

\begin{table}[htbp]
\centering
\caption{Direct quadrature validation for the first positive zero of
$\Xi$, with $\rho=0.5$, $d=0.20$, and $m=4$.}
\label{tab:xi-quadrature-validation}
\begin{tabular}{rcc}
\toprule
$N$ & $E_\mu$ & $E_H$\\
\midrule
16 & $2.822\times10^{-4}$ & $2.822\times10^{-4}$\\
32 & $7.959\times10^{-8}$ & $7.959\times10^{-8}$\\
64 & $5.496\times10^{-15}$ & $4.759\times10^{-15}$\\
\bottomrule
\end{tabular}
\end{table}

In these calculations, the distance from the nearest zero to the contour is the dominant resolution parameter. The quadrature order in
the moving-contour test is therefore increased as a sampled contour
approaches a crossing event.

\paragraph{Indicator fields and Hankel signatures.}
Using the exact atomic representation of the first three positive
zeros of $\Xi$, we evaluate the fields
$$
\mathcal N_\rho(z),
\qquad
\mathcal E_{m,\rho}(z),
\qquad
\mathcal V_{m,\rho}(x)
$$
defined in Section~\ref{sec:indicator-functionals}. The count and
energy fields are computed on a grid of complex centers, while the PSD
violation is evaluated along the real center axis. The results are
shown in \Cref{fig:indicator-fields-compact}. The count field is
constant on each region determined by the crossing circles. The energy
field varies within an active region because it retains the normalized
coordinates of the enclosed zeros. Along the real axis, the PSD field
remains at the numerical zero level for the tested all-real
configuration.

\begin{figure}[htbp]
\centering
\begin{subfigure}{0.32\textwidth}
\centering
\includegraphics[width=\textwidth]
{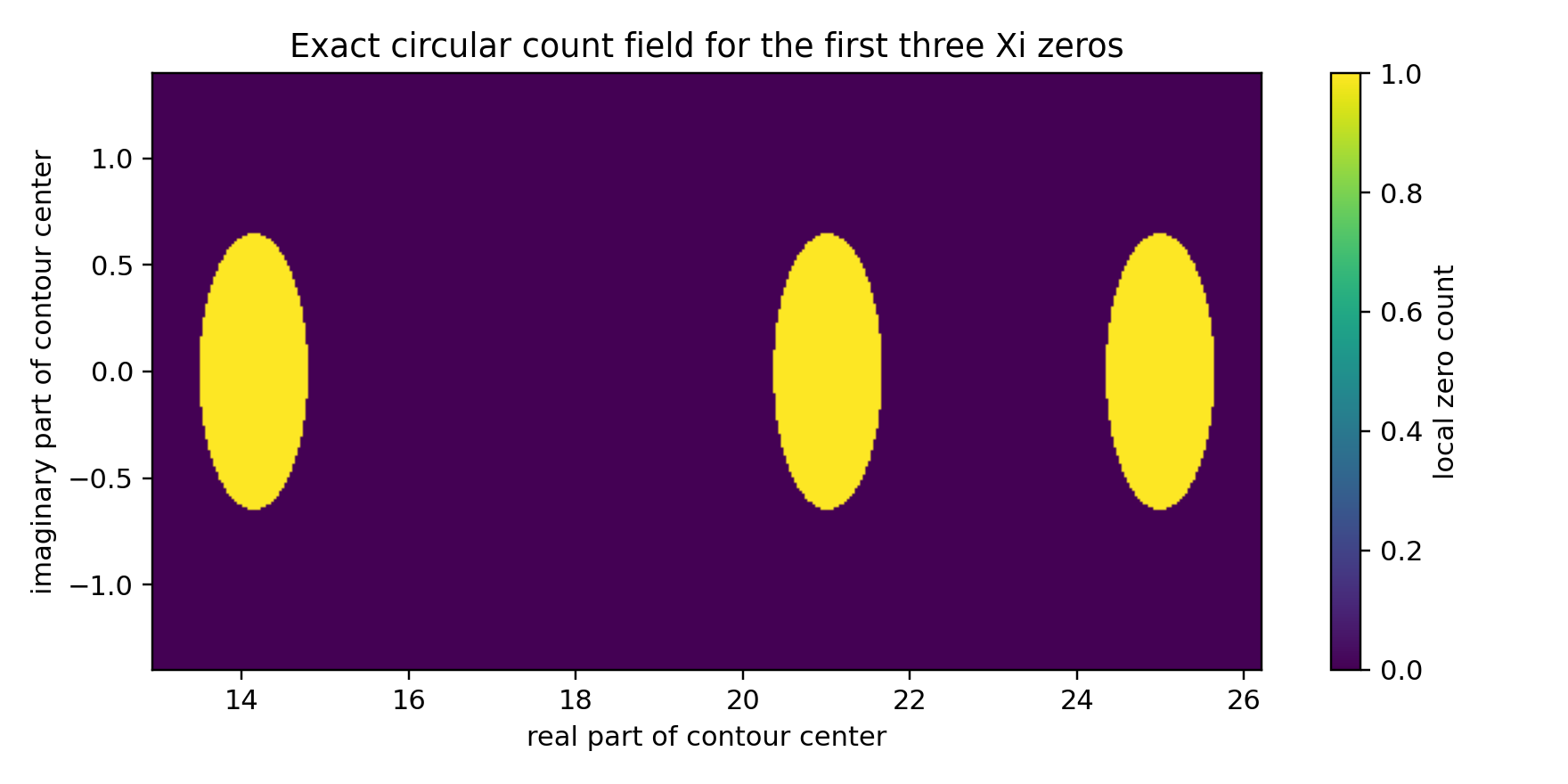}
\caption{Count field.}
\end{subfigure}
\hfill
\begin{subfigure}{0.32\textwidth}
\centering
\includegraphics[width=\textwidth]
{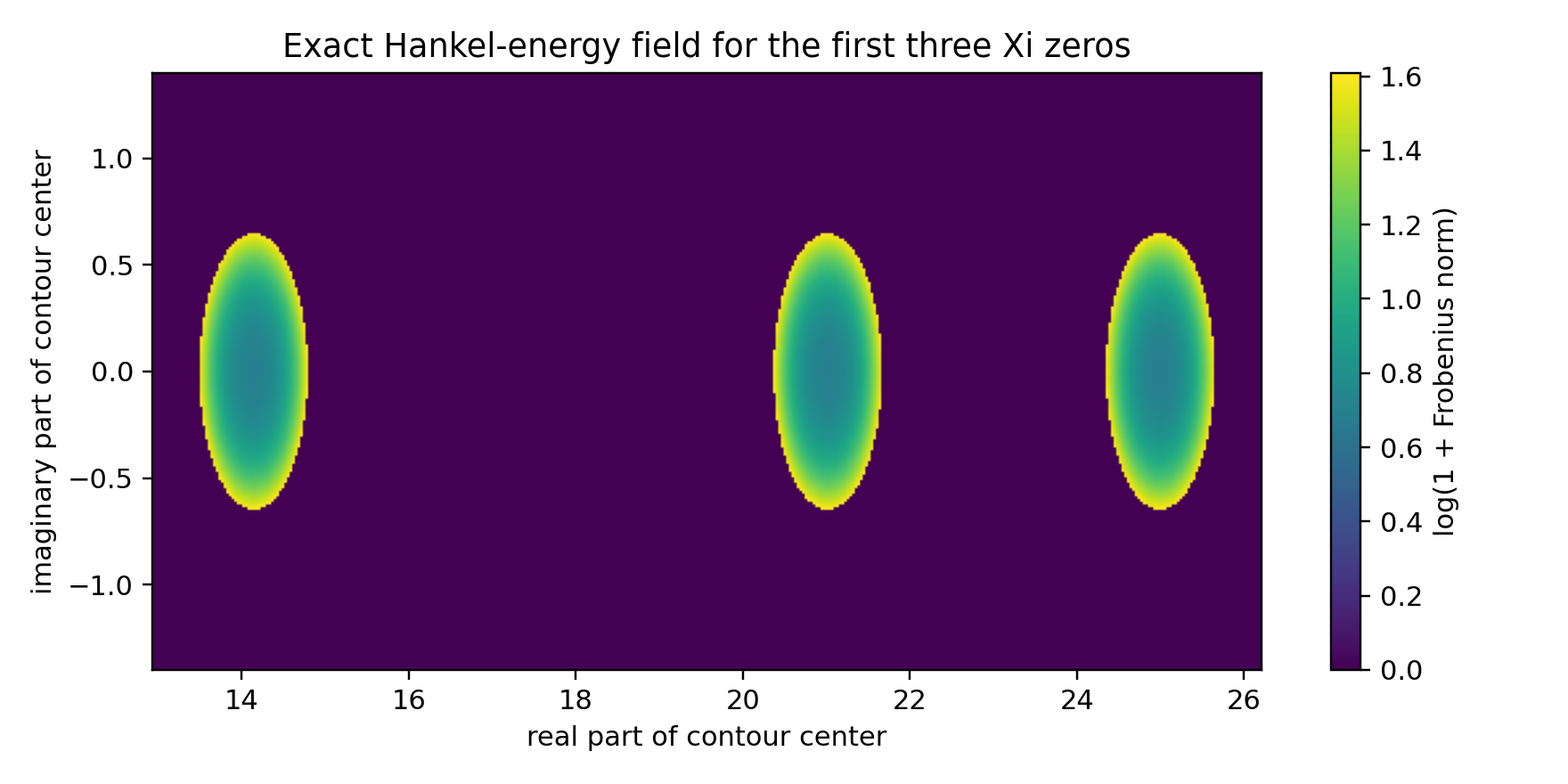}
\caption{Hankel-energy field.}
\end{subfigure}
\hfill
\begin{subfigure}{0.32\textwidth}
\centering
\includegraphics[width=\textwidth]
{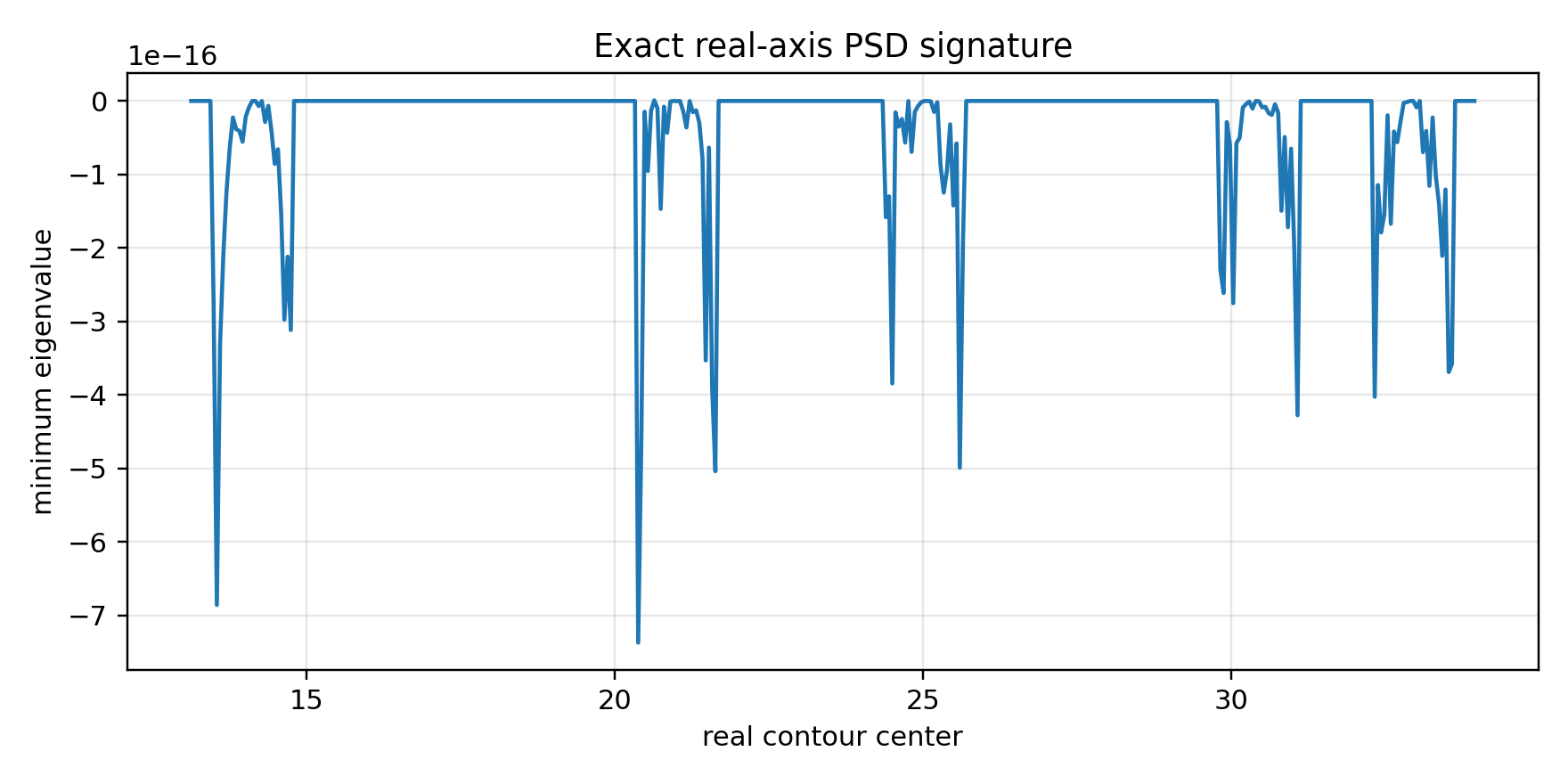}
\caption{Real-axis PSD violation.}
\end{subfigure}
\caption{Count, Hankel-energy, and PSD-violation fields generated from
known positive zeros of $\Xi$.}
\label{fig:indicator-fields-compact}
\end{figure}

Selected centers are also recomputed directly from $\Xi'/\Xi$. For
off-real strips with
$$
|\operatorname{Im}z|=1,
\ 1.5,
\ 2,
\qquad
\rho=0.65,
$$
the maximum computed Hankel norms are approximately
$$
3.3\times10^{-14},
\qquad
4.9\times10^{-23},
\qquad
9.4\times10^{-30},
$$
respectively, consistently with the empty-contour criterion.

For a real-centered circle containing the first two positive zeros,
the exact $5\times5$ centered normalized Hankel matrix has rank two.
At $N=32$, insufficient quadrature resolution produces a spurious
negative eigenvalue of size $7.5\times10^{-8}$. At $N=64$, the
minimum eigenvalue is $-6.4\times10^{-16}$, and the numerical rank
is correctly identified as two. These tests show that small negative eigenvalues must be
interpreted relative to quadrature accuracy.

The centered normalized coordinate is also important for numerical scaling. For a rank-$r$ Hankel matrix, we report the effective condition number
$$
\kappa_{\mathrm{eff}}(H_m)
=
\frac{\sigma_1(H_m)}{\sigma_r(H_m)},
$$
which excludes the exact zero singular values caused by rank deficiency. For the twentieth and twenty-first positive zeros, the raw-coordinate Hankel matrices have effective condition numbers of approximately $3.1\times10^7$--$3.3\times10^7$, and their largest entries grow to approximately $2.5\times10^{34}$ at $m=10$. After centering and normalization, the effective condition number remains below approximately $2.1$, and the largest entry is $2$ in the same test. Thus affine normalization preserves the exact rank and inertia while substantially improving numerical scaling.

\paragraph{Continuous transport and event recovery.}
We finally test the complete moving-contour mechanism for the same zero
$\gamma_1$. The contour radius is fixed at $\rho=0.6$, the center
$x$ moves along the real axis, and $m=4$. The entering event occurs
at
$$
x_*=\gamma_1-\rho.
$$

On a crossing-free interval beginning at $x_0=x_*+0.20$, the initial
matrix is computed by direct contour quadrature and propagated by
\begin{equation}
H_m^{\mathrm{flow}}(x)
=
\Phi_m(x,x_0)
H_m(x_0)
\Phi_m(x,x_0)^T.
\label{eq:numerical-flow-prediction}
\end{equation}
We compare it with an independently recomputed contour matrix through
\begin{equation}
R_{\mathrm{flow}}(x)
=
\frac{
\|H_m^{\mathrm{quad}}(x)-H_m^{\mathrm{flow}}(x)\|_F
}{
1+\|H_m^{\mathrm{quad}}(x)\|_F
}.
\label{eq:numerical-flow-residual}
\end{equation}
For center displacements between $0$ and $0.20$, with $N=128$,
we obtain
$$
\max_x R_{\mathrm{flow}}(x)
=
1.71\times10^{-15}.
$$
Thus independently computed contour states satisfy the continuous
congruence law to the double-precision level.

To test event recovery without assuming prior knowledge of $x_*$,
let $[x_L,x_R]$ be a sampled interval of length
$h=x_R-x_L$ in which the rounded zeroth moment changes from zero to
one, and set $x_c=(x_L+x_R)/2$. Transporting the endpoint matrices to
the common midpoint coordinate gives
$$
\widetilde H_L
=
\Phi_m(x_c,x_L)H_m(x_L)\Phi_m(x_c,x_L)^T,
$$
$$
\widetilde H_R
=
\Phi_m(x_c,x_R)H_m(x_R)\Phi_m(x_c,x_R)^T,
$$
and the drift-corrected jump is
\begin{equation}
\widehat Q
=
\widetilde H_R-
\widetilde H_L.
\label{eq:numerical-transported-jump}
\end{equation}
For a single entering zero, the multiplicity, midpoint coordinate, and
zero location are recovered from
\begin{equation}
\widehat n
=
\widehat Q_{00},
\qquad
\widehat u
=
\frac{\widehat Q_{10}}{\widehat Q_{00}},
\qquad
\widehat\tau
=
x_c+\rho\widehat u.
\label{eq:numerical-event-recovery}
\end{equation}
These formulas do not use an estimate of the exact crossing position
inside the sampled interval.

For benchmarking, the true crossing is placed at $37\%$ of the
interval, so the midpoint bracketing error is $0.13h$. As $h$ is
reduced, the quadrature order is increased from $128$ to $512$.
Let
$$
E_Q
=
\frac{\|\widehat Q-Q\|_F}{\|Q\|_F},
\qquad
E_n
=
|\widehat n-1|,
\qquad
E_\tau
=
|\widehat\tau-\gamma_1|,
$$
where $Q$ is the exact rank-one jump represented in the midpoint
coordinate. The results are summarized in
\Cref{fig:dynamic-validation-compact,tab:event-recovery-compact}.

\begin{figure}[htbp]
\centering
\begin{subfigure}{0.48\textwidth}
\centering
\includegraphics[width=\textwidth]
{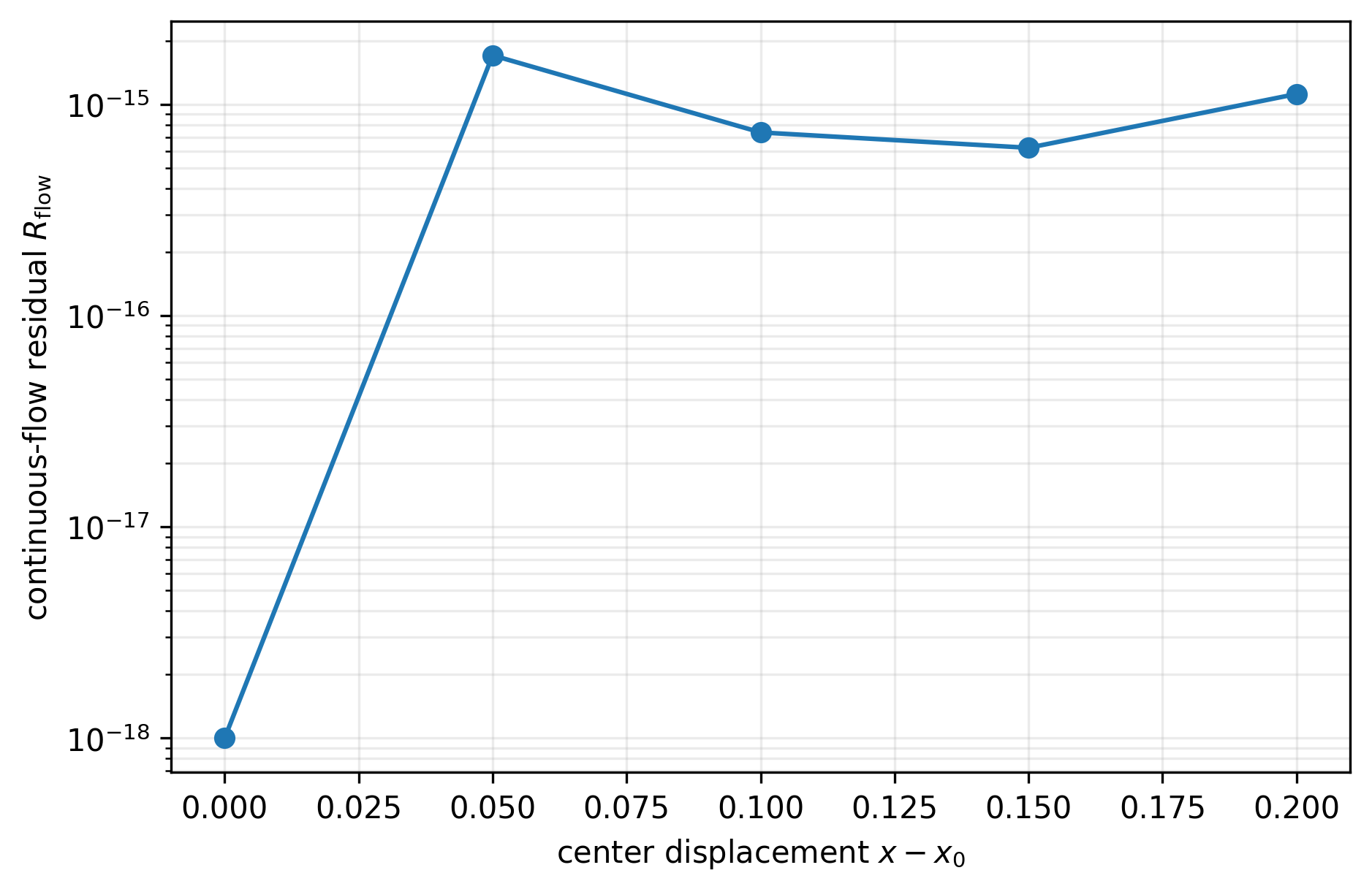}
\caption{Continuous-flow residual.}
\end{subfigure}
\hfill
\begin{subfigure}{0.48\textwidth}
\centering
\includegraphics[width=\textwidth]
{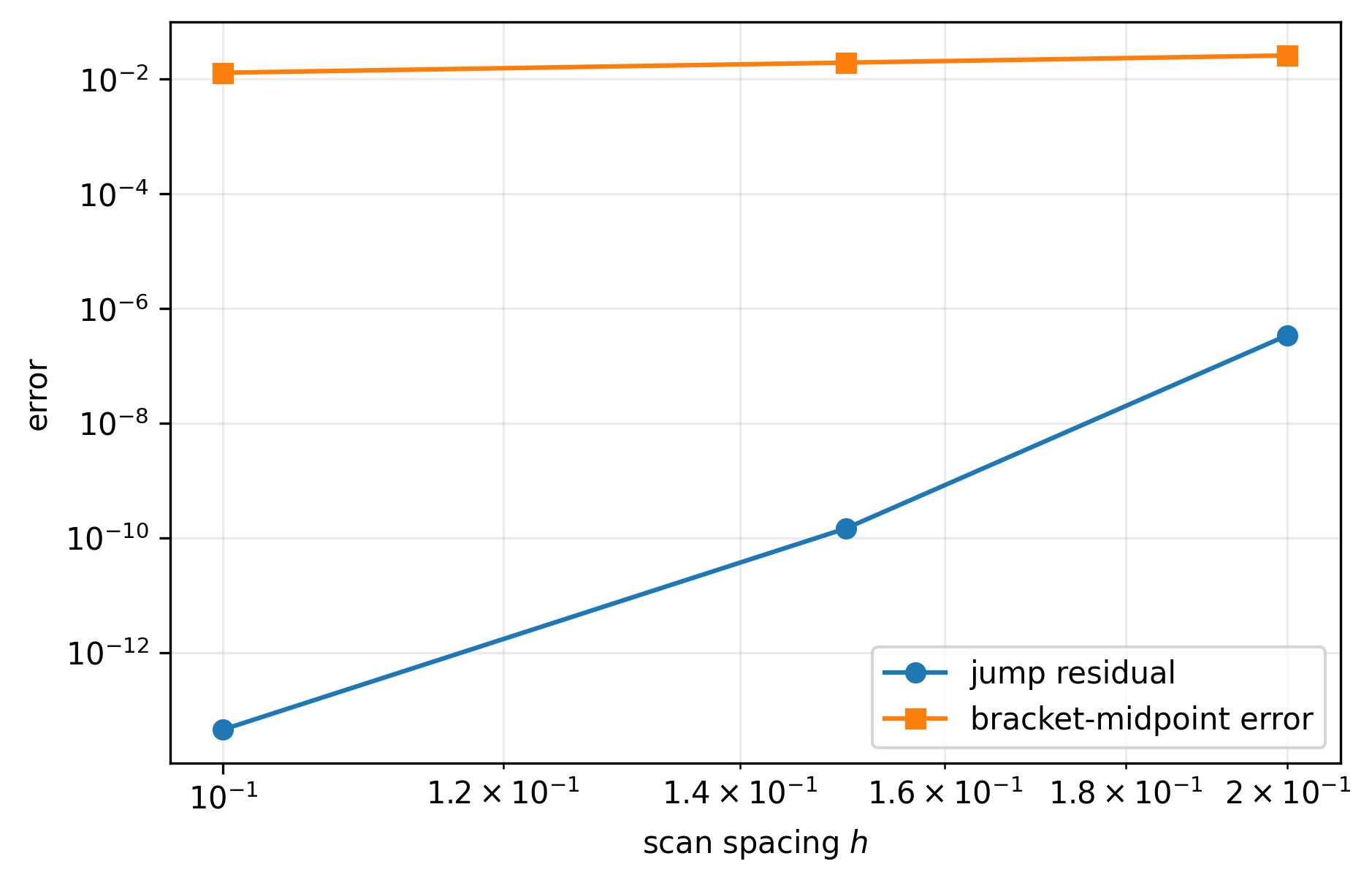}
\caption{Sampled crossing and transported jump.}
\end{subfigure}
\caption{Continuous transport and event recovery for the first
positive zero of $\Xi$.}
\label{fig:dynamic-validation-compact}
\end{figure}

\begin{table}[htbp]
\centering
\caption{Recovery of the first positive zero from a sampled entering
event.}
\label{tab:event-recovery-compact}
\small
\setlength{\tabcolsep}{4pt}
\begin{tabular}{rrcccc}
\toprule
$h$ & $N$ & $E_Q$ &
$\sigma_2(\widehat Q)/\sigma_1(\widehat Q)$ &
$E_n$ & $E_\tau$\\
\midrule
0.20 & 128 & $3.427\times10^{-7}$ &
$8.664\times10^{-17}$ &
$3.427\times10^{-7}$ &
$3.14\times10^{-14}$\\
0.15 & 256 & $1.459\times10^{-10}$ &
$4.520\times10^{-17}$ &
$1.459\times10^{-10}$ &
$3.14\times10^{-14}$\\
0.10 & 512 & $4.542\times10^{-14}$ &
$5.433\times10^{-17}$ &
$4.952\times10^{-14}$ &
$3.14\times10^{-14}$\\
\bottomrule
\end{tabular}
\end{table}

The singular-value ratio remains at the $10^{-16}$ level in all
three runs, confirming the predicted rank-one structure. Although the
midpoint localizes the crossing only to an error proportional to
$h$, the ratio in \eqref{eq:numerical-event-recovery} recovers the
zero to the double-precision level. The jump and multiplicity errors
decrease from approximately $3.4\times10^{-7}$ to
$5.0\times10^{-14}$.

Taken together, the experiments validate the principal computational
features of the framework. The contour matrices are accurately
computed when the integration path remains sufficiently separated from
the zeros; the indicator fields distinguish count, coordinate, and
inertia information; and the moving-contour trajectory decomposes
numerically into a continuous congruence flow and isolated low-rank
events. The main limitation observed in these tests is the deterioration of contour quadrature when a pole of $\Xi'/\Xi$ approaches the integration path.
Here this effect is controlled by coupling the scan spacing with the
quadrature order.

\section{Conclusion}
\label{sec:conclusion}

This paper develops a finite-dimensional matrix representation of the local zero geometry of the Riemann $\Xi$-function by combining contour moments, atomic measures, and Hankel matrices. The resulting framework unifies static and dynamical information: the vanishing and stabilized rank of the Hankel matrices describe the enclosed coordinate support, while their inertia distinguishes real nodes from nonreal conjugate pairs and leads to a local Hankel-positivity formulation of the Riemann hypothesis. When the contour moves, the same matrices satisfy a continuous congruence flow interrupted by finite-rank crossing events; after correction for the known coordinate drift, the trajectory becomes piecewise constant and its jumps encode crossing directions, multiplicities, coordinates, and zero locations. The numerical experiments confirm the predicted flow and jump structures and show that centered normalization is important for stable computation, with near-contour zeros remaining the principal source of quadrature ill-conditioning. The framework should therefore be viewed as a local structural and computational reformulation of zero configurations, rather than as a proof of the Riemann hypothesis. The central unresolved problem is to derive the required Hankel positivity directly from the analytic or arithmetic structure of $\Xi$.

\bibliographystyle{plain}
\bibliography{myref_revised}

\end{document}